\documentclass[reqno]{amsart}

\usepackage{mathrsfs, color}
\usepackage{amsfonts}
\usepackage{amssymb,amsmath}
\usepackage{amsthm}
\usepackage{graphics}
\usepackage{graphicx}
\usepackage{epsfig}

\newtheorem{theorem}{Theorem}[section]
\newtheorem{lemma}[theorem]{Lemma}

\theoremstyle{definition}

\theoremstyle{remark}
\newtheorem{remark}[theorem]{Remark}

\numberwithin{equation}{section}

\usepackage{algorithm} 
\usepackage{multirow} 
\usepackage{xcolor}
\usepackage{makecell}
\usepackage{lipsum}
\usepackage{amsfonts}
\usepackage{graphicx}
\usepackage{epstopdf}
\usepackage{algorithmic}
\ifpdf
  \DeclareGraphicsExtensions{.eps,.pdf,.png,.jpg}
\else
  \DeclareGraphicsExtensions{.eps}
\fi

\begin{document}

\title[Model error learning method for IMSP]
 {Recursive linearization method for inverse medium scattering problems with complex mixture Gaussian error learning}

\author[J.X.Jia]{Junxiong Jia}
\address{School of Mathematics and Statistics,
Xi'an Jiaotong University,
 Xi'an
710049, China}
\email{jjx323@xjtu.edu.cn}
\thanks{}

\author[B. Wu]{Bangyu Wu}
\address{School of Mathematics and Statistics,
Xi'an Jiaotong University,
Xi'an,
710049, China}
\email{bangyuwu@xjtu.edu.cn}

\author[J. Peng]{Jigen Peng}
\address{School of Mathematics and Statistics,
Xi'an Jiaotong University,
 Xi'an
710049, China}
\email{jgpeng@xjtu.edu.cn}

\author[J. Gao]{Jinghuai Gao}
\address{School of Electronic and Information Engineering,
Xi'an Jiaotong University,
 Xi'an
710049, China}
\email{jhgao@xjtu.edu.cn}


\date{}


\begin{abstract}
This paper is concerned with the modeling errors appeared in the numerical methods of inverse medium scattering problems (IMSP).
Optimization based iterative methods are wildly employed to solve IMSP, which are computationally intensive due to a series of
Helmholtz equations need to be solved numerically. Hence, rough approximations of Helmholtz equations can significantly
speed up the iterative procedure. However, rough approximations will lead to instability and inaccurate estimations.
Using the Bayesian inverse methods, we incorporate the modelling errors brought by the rough approximations.
Modelling errors are assumed to be some complex Gaussian mixture (CGM) random variables, and in addition, well-posedness of IMSP in
the statistical sense has been established by extending the general theory to involve CGM noise. Then, we generalize the
real valued expectation-maximization (EM) algorithm used in the machine learning community
to our complex valued case to learn parameters in the CGM distribution.
Based on these preparations, we generalize the recursive linearization method (RLM) to
a new iterative method named as Gaussian mixture recursive linearization method (GMRLM) which takes modelling errors into account.
Finally, we provide two numerical examples to illustrate the effectiveness of the proposed method.
\end{abstract}

\maketitle



\section{Introduction}

Scattering theory has played a central role in the field of mathematical physics, which is concerned with the effect
that an inhomogeneous medium has on an incident particle or wave \cite{ColtonThirdBook}.
Usually, the total field is viewed as the sum of an incident field and a scattered field.
Then, the inverse scattering problems focus on determining the nature of the inhomogeneity from
a knowledge of the scattered field \cite{Bleistein2001Book,ColtonSIAMReview2000}, which
have played important roles in diverse scientific areas
such as radar and sonar, geophysical exploration, medical imaging and nano-optics.

Deterministic computational methods for inverse scattering problems can be classified into two categories:
nonlinear optimization based iterative methods \cite{Bao2015TopicReview,Metivier2016IP,Natterer1995IP}
and imaging based direct methods \cite{Cakoni2006Book,Cheney2001IP}.
Direct methods are called qualitative methods which need no direct solvers and visualize the scatterer by highlighting
its boundary with designed imaging functions.
Iterative methods are usually called quantitative methods, which aim at providing some functions to represent the scatterer.
Because a sequence of direct and adjoint scattering problems need to be solved, the quantitative methods are computationally intensive.

This paper is concerned with the nonlinear optimization based iterative methods, especially focus
on the recursive linearization method (RLM) for inverse medium scattering problems \cite{Bao2015TopicReview}.
Although the computational obstacle can be handled in some circumstances,
the accuracy of the forward solver is still a critical topic, particularly for applications in seismic exploration \cite{Fichtner2011Book}
and medical imaging \cite{Koponen2014IEEE}.
A lot of efficient forward solvers based on finite difference method, finite element methods and spectral methods
have been proposed \cite{Teresa2006IP,Wang1997JASA}.
Here, we will not propose a new forward solver to reduce the computational load, but attempt to reformulate the nonlinear optimization model based
on Bayesian inverse framework which can incorporate statistical properties of the model errors induced by rough forward solvers.
By using the statistical properties, we aim to reduce the computational load for the inverse procedure.

In order to give a clear sketch of our idea, let us provide a concise review of the Bayesian inverse methods according to our purpose.
Denote $X$ to be some separable Banach space, then the forward problem usually modeled as follows
\begin{align}\label{forwardForm}
d = \mathcal{F}(m) + \epsilon,
\end{align}
where $d \in \mathbb{C}^{N_{d}}$ ($N_{d} \in \mathbb{N}^{+}$) stands for the measured data, $m \in X$ represents the interested parameter and
$\epsilon$ denotes noise.
For inverse scattering problems, $m$ is just the scatterer, $\mathcal{F}$ represents a Helmholtz equation combined with some
measurement operator.
The nonlinear optimization based iterative methods just formulate inverse problem as follows
\begin{align}\label{optimiFormu}
\min_{m \in X} \Bigg\{ \frac{1}{2}\big\|d - \mathcal{F}(m)\big\|_{2}^{2} + \mathcal{R}(m) \Bigg\},
\end{align}
where $\mathcal{R}(\cdot)$ stands for some regularization operator and $\|\cdot\|_{2}$ represents the $\ell^{2}$-norm.

Different to the minimization problem (\ref{optimiFormu}), Bayesian inverse methods reformulate the inverse problem as a stochastic inference problem,
which has the ability to give uncertainty quantifications
\cite{inverse_fluid_equation,Besov_prior,Junxiong2016IP,book_comp_bayeisn,acta_numerica}.
Bayesian inverse methods aim to provide a complete posterior information, however, it can also offer a point estimate.
Up to now, there are usually two frequently used point estimators:
maximum a posteriori (MAP) estimate and conditional mean (CM) estimate \cite{Tenorio2006Book}.
For problems defined on finite dimensional space, MAP estimate is obviously just the solution of the minimization problem (\ref{optimiFormu}),
which is illustrated rigorously in \cite{book_comp_bayeisn}.
Different to the finite dimensional case, only recently, serious results for relationships between MAP estimates and
minimization problem (\ref{optimiFormu}) are obtained in \cite{Burger2014IP,MAPSmall2013,Dunlop2016IP} when $X$ is an infinite dimensional space.
Simply speaking, if minimization problem (\ref{optimiFormu}) has been used to solve our inverse problem,
then an assumption has been made that is the noise $\epsilon$ is sampled from some Gaussian distribution $\mathcal{N}(\bar{\epsilon},\Sigma_{\epsilon})$
with mean $\bar{\epsilon}$ and covariance operator $\Sigma_{\epsilon}$.

In real world applications, we would like to use a fast forward solver (limited accuracy) to obtain an estimation as accurately as possible.
Hence, the noise usually not only brought by inaccurate measurements but also induced by
a rough forward solver and inaccurate physical assumptions \cite{Calvetti2017}.
Let us denote $\mathcal{F}_{a}(\cdot)$ to be the forward operator related to some rough forward solver, then (\ref{forwardForm}) can be rewrite as follows by following the methods used in \cite{Koponen2014IEEE}
\begin{align}\label{forwardForm2}
d = \mathcal{F}_{a}(m) + (\mathcal{F}(m) - \mathcal{F}_{a}(m)) + \epsilon.
\end{align}
By denoting $\xi := (\mathcal{F}(m) - \mathcal{F}_{a}(m))$, we obtain that
\begin{align}\label{forwardForm3}
d = \mathcal{F}_{a}(m) + \xi + \epsilon.
\end{align}
From the perspective of Bayesian methods, we can model $\xi$ as a random variable which obviously has the following two important features
\begin{enumerate}
  \item $\xi$ depend on the unknown function $m$;
  \item $\xi$ may distributed according to a complicated probability measure.
\end{enumerate}

For feature (1), we can relax this tough problem to assume that $\xi$ is independent of $m$ but the probability distribution
of $\xi$ and the prior probability measure of $m$ are related with each other \cite{Lasanen2012IPI}.
For feature (2), to the best of our knowledge, the existing literatures only provide a compromised methods that is
assume $\xi$ sampled from some Gaussian probability distributions \cite{Junxiong2016,Koponen2014IEEE}.
Here, we attempt to provide a more realistic assumptions for the probability measures of the random variable $\xi$.

Noticing that Bayes' formula is also one of fundamental tools for investigations about statistical machine learning \cite{PR2006Book} which is a field
attracts numerous researchers from various fields, e.g., computer science, statistics and mathematics.
Notice that for problems such as background subtraction \cite{Yong2017IEEE}, low-rank matrix factorization \cite{Zhao2015IEEE}
and principle component analysis \cite{MENG2012487,Zhao2014ICML},
learning algorithms deduced by Bayes' formula are useful and the errors brought by inaccurate forward modeling also appears.
For modeling errors appeared in machine learning tasks, Gaussian mixture model is widely used since it can
approximate any probability measure in some sense \cite{PR2006Book}.

Gaussian mixture distributions usually have the following form of density function
\begin{align}
\sum_{k = 1}^{K}\pi_{k} \mathcal{N}(\cdot \,| \,\zeta_{k},\Sigma_{k}),
\end{align}
where $\mathcal{N}(\cdot \,| \,\zeta_{k},\Sigma_{k})$ stands for a Gaussian probability density function with
mean value $\zeta_{k}$ and covariance matrix $\Sigma_{k}$ and for every $k$, $\pi_{k}\in (0,1)$ satisfy
$\sum_{k=1}^{K}\pi_{k} = 1$.
In the following, we always assume that the measurement noise $\epsilon$ is a Gaussian random variable with mean $0$
and covariance matrix $\nu I$ ($\nu \in \mathbb{R}^{+}$ and $I$ is an identity matrix).
For our problem, we can intuitively provide the following optimization problem if we assume $\xi$
sampled from some Gaussian mixture probability distributions
\begin{align}\label{modelQ1}
\min_{m\in X}\Bigg\{-\ln\Big( \sum_{k = 1}^{K}\pi_{k} \mathcal{N}(d-\mathcal{F}_{a}(m) \,| \, \zeta_{k},\Sigma_{k} + \nu I) \Big) + \mathcal{R}(m) \Bigg\}.
\end{align}

In the machine learning field, there usually have a lot of sampling data and the forward problems are not computationally intensive
compared with the inverse medium scattering problem. Hence, they use alternative iterative methods to
find the optimal solution and estimate the modeling error simultaneously \cite{Zhao2015IEEE}.
However, considering the lack of learning data and the high computational load of our forward problems,
we can not trivially generalize their alternative iterative methods to our case.
In order to employ Gaussian mixture distribution, we will meet the following three problems
\begin{enumerate}
  \item Under which conditions, Bay's formula and MAP estimate with Gaussian mixture distribution hold in infinite-dimensional space;
  \item How to construct learning examples and how to learn the parameters in Gaussian mixture distributions.
  Firstly, since we can hardly have so many learning examples as for the usual machine learning problem, we will meet a situation that is
  the number of learning examples are smaller than the number of discrete points which is also an ill-posed problem.
  Secondly, the solution of Helmholtz equation is a complex valued function. Because of that, we should develop learning
  algorithms for complex valued variables which is different to the classical cases for machine learning tasks \cite{PR2006Book,Yong2017IEEE}.
  \item For the complicated minimization problem (\ref{modelQ1}), how to construct a suitable iterative type method, i.e., some modified RLM.
\end{enumerate}

In this paper, we provide a primitive study about these three problems. Theoretical foundations for using
Gaussian mixture distributions in infinite-dimensional space problems have been established.
Learning algorithm has been designed based on the relationships between real Gaussian distribution and
complex Gaussian distribution. By carefully calculations, a modified RLM name as Gaussian mixture recursive linearization method (GMRLM)
has been proposed to efficiently solve the inverse medium problem with multi-frequencies data.
Numerical examples are finally reported to illustrate the effectiveness of the proposed method.

The outline of this paper is as follows.
In Section 2, general Bayesian inverse method with Gaussian mixture noise model is established and
the relationship between MAP estimators with classical regularization methods is also discussed.
In Section 3, well-posedness of inverse medium scattering problem in the Bayesian sense is proved.
Then, we propose the learning algorithm for Gaussian mixture distribution by generalizing the real valued
expectation-maximization (EM) algorithm to complex valued EM algorithm.
At last, we deduce the Gaussian mixture recursive linearization method.
In Section 4, two typical numerical examples are given, which illustrate the effectiveness of the proposed methods.


\section{Bayesian inverse theory with Gaussian mixture distribution}\label{BayeTheoSection}

In this section, we prove the well-posedness and illustrate the validity of MAP estimate of inverse problems under the Bayesian inverse framework
when the noise is assumed to be a random variable sampled from a complex valued Gaussian mixture distribution.

Before diving into the main contents, let us provide a brief notation list which
will be used in all of the following parts of this paper.

\textbf{Notations:}
\begin{itemize}
  \item For an integer $N$, denote $\mathbb{C}^{N}$ as $N$-dimensional complex vector space;
  $\mathbb{R}^{+}$ and $\mathbb{N}^{+}$ represent positive real numbers and positive integers respectively;
  \item For a Banach space $X$, $\|\cdot\|_{X}$ stands for the norm defined on $X$ and, particularly, $\|\cdot\|_{2}$ represents
  the $\ell^{2}$-norm of $\ell^{2}$ space.
  \item For a matrix $\Sigma$, denote its determinant as $\det(\Sigma)$;
  \item Denote $B(m,R)$ as a ball with center $m$ and radius $R$. Particularly, denote $B_{R} := B(0,R)$ when the ball is centered at origin;
  \item Denote $X$ and $Y$ to be some Banach space;
  For an operator $F :\, X \rightarrow Y$, denote $F'(x_{0})$ as the Fr\'{e}chet derivative of $F$ at $x_{0} \in X$.
  \item Denote $\text{Re}(\xi)$, $\text{Imag}(\xi)$, $\xi^{T}$, $\xi^{H}$ and $\bar{\xi}$ as the real part, imaginary part, transpose, conjugate transpose
  and complex conjugate of $\xi \in \mathbb{C}^{N}$ respectively;
  \item The notation $\eta \sim p(\eta)$ stands for a random variable $\eta$ obeys the
  probability distribution with density function $p(\cdot)$.
\end{itemize}

Let $\mathcal{N}_{c}(\eta \,|\, \zeta,\Sigma)$ represents the density function of
$N_{d}$-dimensional complex valued Gaussian probability distribution \cite{Goodman1963Annals} defined as follows
\begin{align}
\mathcal{N}_{c}(\eta \, | \, \zeta,\Sigma) :=
 \frac{1}{(\pi)^{N_{d}}\det(\Sigma)}
\exp\left( -\Big\|\eta-\zeta\Big\|_{\Sigma}^{2} \right),
\end{align}
where $\zeta$ is a $N_{d}$-dimensional complex valued vector, $\Sigma$ is a positive definite Hermitian matrix
and $\|\cdot\|_{\Sigma}^{2}$ is defined as follow
\begin{align}
\big\|\eta-\zeta\big\|_{\Sigma}^{2} :=
\big(\eta-\zeta\big)^{H} \, \Sigma^{-1} \, \big(\eta-\zeta\big),
\end{align}
with the superscript $H$ stands for conjugate transpose.
Denote $\eta := \xi + \epsilon$, then formula (\ref{forwardForm3}) can be written as follows
\begin{align}
d = \mathcal{F}_{a}(m) + \eta,
\end{align}
where
\begin{align}
d \in \mathbb{C}^{N_{d}}, \quad \eta \sim \sum_{k = 1}^{K}\pi_{k}\mathcal{N}_{c}(\eta \,|\, \zeta_{k},\Sigma_{k} + \nu I),
\end{align}
with $N_{d}$, $K$ denote some positive integers and $\nu \in \mathbb{R}^{+}$.

Before going further, let us provide the following basic assumptions about the approximate forward operator $\mathcal{F}_{a}$.

\textbf{Assumption 1.}
\begin{enumerate}
  \item for every $\epsilon > 0$ there is $M = M(\epsilon) \in \mathbb{R}$, $C\in\mathbb{R}$ such that, for all $m \in X$,
  \begin{align*}
  \|\mathcal{F}_{a}(m)\|_{2} \leq C \exp(\epsilon\|m\|_{X}^{2} + M).
  \end{align*}
  \item for every $r > 0$ there is $K = K(r) > 0$ such that, for all $m \in X$ with
  $\|m\|_{X} < r$, we have
  \begin{align*}
  \|\mathcal{F}_{a}'(m)\|_{op} \leq K,
  \end{align*}
  where $\|\cdot\|_{op}$ denotes the operator norm.
\end{enumerate}

At this stage, we need to provide some basic notations of the Bayesian inverse method when $m$ in some infinite-dimensional space.
Following the work \cite{inverse_fluid_equation,acta_numerica}, let $\mu_{0}$ stands for the prior probability measure defined on
a separable Banach space $X$ and denote $\mu^{d}$ to be the posterior probability measure.
Then the Bayes' formula may be written as follows
\begin{align}
\frac{d\mu^{d}}{d\mu_{0}}(m) & = \frac{1}{Z(d)} \exp\Big( \Phi(m;d) \Big), \label{DefineMuY} \\
Z(d) & = \int_{X} \exp\Big( \Phi(m;d) \Big)\mu_{0}(dm),  \label{DefineOfZd}
\end{align}
where $\frac{d\mu^{d}}{d\mu_{0}}(\cdot)$ represents the Radon-Nikodym derivative and
\begin{align}
\Phi(m;d) := \ln\Bigg\{\sum_{k = 1}^{K} \pi_{k} \frac{1}{\pi^{N_{d}}\det(\Sigma_{k} + \nu I)}
\exp\left( -\Big\|d-\mathcal{F}_{a}(m)-\zeta_{k}\Big\|_{\Sigma_{k} + \nu I}^{2} \right) \Bigg\}.
\end{align}


\subsection{Well-posedness}

In this subsection, we prove the following results which demonstrate formula (\ref{DefineMuY}) and (\ref{DefineOfZd}) under
some general conditions.
\begin{theorem}\label{wellPosedBaye}
Let Assumption 1 holds for some $\epsilon$, $r$, $K$ and $M$. Assume that $X$ is some separable Banach space, $\mu_{0}(X) = 1$ and
that $\mu_{0}(X\cap B) > 0$ for some bounded set $B$ in $X$.
In addition, we assume $\int_{X} \exp(2\epsilon \|m\|_{X}^{2}) \mu_{0}(dm) < \infty$.
Then, for every $d \in \mathbb{C}^{N_{d}}$,
$Z(d)$ given by (\ref{DefineOfZd}) is positive and the probability measure $\mu^{d}$ given by (\ref{DefineMuY})
is well-defined. In addition, there is $C = C(r) > 0$ such that, for all $d_{1}, d_{2} \in B(0,r)$
\begin{align*}
d_{\text{Hell}}(\mu^{d_{1}}, \mu^{d_{2}}) \leq C \|d_{1} - d_{2}\|_{2},
\end{align*}
where $d_{\text{Hell}}(\cdot,\cdot)$ denotes the Hellinger distance defined for two probability measures.
\end{theorem}
\begin{proof}
In order to prove this theorem, we need to verify three conditions stated in Assumption 4.2 and Theorem 4.4 in \cite{Dashti2014}.
Since
\begin{align*}
\sum_{k = 1}^{K} \pi_{k} \frac{1}{\pi^{N_{d}}\det(\Sigma_{k}+\nu I)}
\exp\left( -\Big\|d-\mathcal{F}_{a}(m)-\zeta_{k}\Big\|_{\Sigma_{k}+\nu I}^{2} \right) \leq 1,
\end{align*}
we know that
\begin{align}\label{Cond1}
\Phi(m;d) \leq 0.
\end{align}
In the following, we denote
\begin{align*}
f_{k}(d,m):= \big(d-\mathcal{F}_{a}(m)-\zeta_{k}\big)^{H} \, \big(\Sigma_{k}+\nu I\big)^{-1} \, \big(d-\mathcal{F}_{a}(m)-\zeta_{k}\big).
\end{align*}
Then, we have
\begin{align*}
\nabla_{d}f_{k}(d,m) & = (d - \mathcal{F}_{a}(m) - \zeta_{k})^{H}(\Sigma_{k}+\nu I)^{-1}
+ \overline{(d-\mathcal{F}_{a}(m)-\zeta_{k})^{H}(\Sigma_{k}+\nu I)^{-1}}    \\
& = 2\text{Re}\Big( (d - \mathcal{F}_{a}(m) - \zeta_{k})^{H}(\Sigma_{k}+\nu I)^{-1} \Big).
\end{align*}
Through some simple calculations, we find that
\begin{align}\label{DdPhi1}
\nabla_{d}\Phi(m;d) = - \sum_{k=1}^{K} 2 g_{k} \text{Re}\Big( (d - \mathcal{F}_{a}(m) - \zeta_{k})^{H}(\Sigma_{k}+\nu I)^{-1} \Big),
\end{align}
where
\begin{align}\label{gkDef}
g_{k} := \frac{\pi_{k}\mathcal{N}_{c}(d-\mathcal{F}_{a}(m) \,|\, \zeta_{k},\Sigma_{k}+\nu I)}
{\sum_{j=1}^{K}\pi_{j}\mathcal{N}_{c}(d-\mathcal{F}_{a}(m) \,|\, \zeta_{j},\Sigma_{j}+\nu I)}.
\end{align}
From the expression (\ref{DdPhi1}) and (i) of Assumption 1, we can deduce that
\begin{align}\label{BoundDd}
\|\nabla_{d}\Phi(m;d)\|_{2} \leq C\big( 1 + \|d\|_{2} + \exp(\epsilon\|m\|_{X}^{2}) \big).
\end{align}
where the constant $C$ depends on $K$, $\{\Sigma_{k}\}_{k=1}^{K}$ and $\{\zeta_{k}\}_{k = 1}^{K}$.
Considering (\ref{BoundDd}), we obtain
\begin{align}\label{Cond2}
|\Phi(m;d_{1}) - \Phi(m;d_{2})| \leq C\big( 1 + r + \exp(\epsilon\|m\|_{X}^{2}) \big) \|d_{1} - d_{2}\|_{2}.
\end{align}
By our assumptions, the following relation obviously hold
\begin{align}\label{Cond3}
C^{2}\big( 1 + r + \exp(\epsilon\|m\|_{X}^{2}) \big)^{2} \in L_{\mu_{0}}^{1}(X;\mathbb{R}).
\end{align}
At this stage, estimates (\ref{Cond1}), (\ref{Cond2}) and (\ref{Cond3}) verify Assumption 4.2 and conditions of Theorem 4.4 in \cite{Dashti2014}.
Employing theories constructed in \cite{Dashti2014}, we complete the proof.
\end{proof}

\begin{remark}
The assumptions of the prior probability measure are rather general, which include Gaussian probability measure
and TV-Gaussian probability measure \cite{TGPrior2016} for certain space $X$.
\end{remark}


\subsection{MAP estimate}
Through MAP estimate, Bayesian inverse method and classical regularization method are in accordance with each other.
Because our aim is to develop an efficient optimization method, we need to demonstrate the validity of MAP estimate which
provide theoretical foundations for our method.

Firstly, let us assume that the prior probability measure $\mu_{0}$ is a Gaussian probability measure and
define the following functional
\begin{align}\label{MiniProForm}
J(m) = \left \{\begin{aligned}
& -\Phi(m;d) + \frac{1}{2} \|m\|_{E}^{2} \quad \text{if }m\in E, \text{ and} \\
& + \infty, \quad\quad\quad\quad\quad\quad\quad\,\,\, \text{else.}
\end{aligned}\right.
\end{align}
Here $(E,\|\cdot\|_{E})$ denotes the Cameron-Martin space associated to $\mu_{0}$.
In infinite dimensions, we adopt small ball approach constructed in \cite{MAPSmall2013}.
For $m \in E$, let $B(m,\delta) \in X$ be the open ball centred at $m \in X$ with radius $\delta$ in $X$.
Then, we can prove the following theorem which encapsulates the idea that probability is maximized where $J(\cdot)$ is minimized.

\begin{theorem}\label{SmallBall}
Let Assumption 1 holds and assume that $\mu_{0}(X) = 1$. Then the function $J(\cdot)$ defined by (\ref{MiniProForm}) satisfies, for any
$m_{1}, m_{2} \in E$,
\begin{align*}
\lim_{\delta\rightarrow 0}\frac{\mu(B(m_{1},\delta))}{\mu(B(m_{2},\delta))} =
\exp\Big( J(m_{2}) - J(m_{1}) \Big).
\end{align*}
\end{theorem}
\begin{proof}
In order to prove this theorem, let us verify the following two conditions concerned with $\Phi(m;d)$,
\begin{enumerate}
  \item for every $r > 0$ there exists $K = K(r) > 0$ such that, for all $m \in X$ with $\|m\|_{X} \leq r$ we have $\Phi(m;d) \geq K$.
  \item for every $r > 0$ there exists $L = L(r) > 0$ such that, for all $m_{1}, m_{2} \in X$ with $\|m_{1}\|_{X}, \|m_{2}\|_{X} < r$
  we have $|\Phi(m_{1};d) - \Phi(m_{2};d)| \leq L \|m_{1} - m_{2}\|_{X}$.
\end{enumerate}
For the first condition, by employing Jensen's inequality, we have
\begin{align*}
\Phi(m;d) & = \ln\Big( \sum_{k = 1}^{K}\pi_{k} \mathcal{N}_{c}\big(d - \mathcal{F}_{a}(m) \, | \, \zeta_{k}, \Sigma_{k}+\nu I\big) \Big) \\
& \geq \sum_{k = 1}^{K} \pi_{k} \ln\Bigg( \frac{1}{\pi^{N_{d}}|\Sigma_{k} + \nu I|}
\exp\bigg( -\big\| d-\mathcal{F}_{a}(m) - \zeta_{k} \big\|_{\Sigma_{k}+\nu I}^{2} \bigg) \Bigg) \\
& \geq \sum_{k = 1}^{K} \pi_{k} \Big( -\big\| d - \mathcal{F}_{a}(m) - \zeta_{k}
\big\|_{\Sigma_{k}+\nu I}^{2} - N_{d}\ln(\pi) - \ln(|\Sigma_{k}+\nu I|) \Big) \\
& \geq - C \big( 1 + \|d\|_{2}^{2} + \exp(\epsilon r^{2}) \big),
\end{align*}
where $C$ is a positive constant depends on $K$, $\{\pi_{k}\}_{k=1}^{K}$, $\{\Sigma_{k}\}_{k = 1}^{K}$, $\{\zeta_{k}\}_{k=1}^{K}$ and $N_{d}$.
Now, the first condition holds true by choosing $K = - C \big( 1 + \|d\|_{2}^{2} + \exp(\epsilon r^{2}) \big)$.

In order to verify the second condition, we denote
\begin{align*}
f_{k}(d,m):= \big(d-\mathcal{F}_{a}(m)-\zeta_{k}\big)^{H} \, \big(\Sigma_{k}+\nu I\big)^{-1} \, \big(d-\mathcal{F}_{a}(m)-\zeta_{k}\big),
\end{align*}
then focus on the derivative of $f_{k}$ with respect to $m$.
Through some calculations, we find that
\begin{align}
\nabla_{m}f_{k}(d,m) = - 2 \text{Re}\Big( (d-\mathcal{F}_{a}(m)-\zeta_{k})^{H}(\Sigma_{k}+\nu I)^{-1}\mathcal{F}_{a}'(m) \Big).
\end{align}
Hence, we have
\begin{align}\label{DePhi1}
\nabla_{m}\Phi(m;d) = - \sum_{k=1}^{K}2 g_{k} \text{Re}\Big( (d-\mathcal{F}_{a}(m)-\zeta_{k})^{H}(\Sigma_{k}+\nu I)^{-1}\mathcal{F}_{a}'(m) \Big),
\end{align}
where $g_{k}$ defined as in (\ref{gkDef}).
Using Assumption 1 and formula (\ref{DePhi1}), we find that
\begin{align}
|\Phi(m_{1};d) - \Phi(m_{2};d)| \leq C K (1+\|d\|_{2} + \exp(\epsilon \, r^{2})) \|m_{1} - m_{2}\|_{X}.
\end{align}
Let $L = C K (1+\|d\|_{2} + \exp(\epsilon \, r^{2}))$, obviously the second condition holds true.
Combining these two conditions with (\ref{Cond1}), we can complete the proof by using Theorem 4.11 in \cite{Dashti2014}.
\end{proof}

Now, if we assume $\mu_{0}$ is a TV-Gaussian probability measure, then we can define the following functional
\begin{align}\label{MiniProFormTG}
J(m) = \left \{\begin{aligned}
& -\Phi(m;d) + \lambda \|m\|_{\text{TV}} + \frac{1}{2} \|m\|_{E}^{2} \quad \text{if }m\in E, \text{ and} \\
& + \infty, \quad\quad\quad\quad\quad\quad\quad\quad\quad\quad\quad\quad\,\, \text{else.}
\end{aligned}\right.
\end{align}
Using similar methods as for the Gaussian case and the above functional (\ref{MiniProFormTG}), we can prove
a similar theorem to illustrate that the MAP estimate is also the minimal solution of
$\min_{m\in X}J(m)$.


\section{Inverse medium scattering problem}\label{SecInverMedi}

In this section, we will apply the general theory developed in Section \ref{BayeTheoSection} to a specific
inverse medium scattering problem. Then we construct algorithms to learn the parameters appeared in the complex Gaussian
mixture distribution. At last, we give the model error compensation based recursive linearization method.
Since the model errors are estimated by some Gaussian mixture distributions, the proposed iterative method
is named as Gaussian mixture recursive linearization method (GMRLM).

Now, let us provide some basic settings of the inverse scattering problem considered in this paper.
In the following, we usually assume that the total field $u$ satisfies
\begin{align}\label{zongEq}
\Delta u + \kappa^{2}(1+q)u = 0 \quad \text{in }\mathbb{R}^{2},
\end{align}
where $\kappa > 0$ is the wavenumber, and $q(\cdot)$ is a real function known as the
scatterer representing the inhomogeneous medium. We assume that the scatterer has a compact support contained
in the ball $B_{R} = \{ \mathbf{r}\in\mathbb{R}^{2} : \, |\mathbf{r}| < R \}$ with boundary
$\partial B_{R} = \{ \mathbf{r}\in\mathbb{R}^{2} : \, |\mathbf{r}| = R \}$, and satisfies
$-1 < q_{\text{min}} \leq q \leq q_{\text{max}} < \infty$, where $q_{\text{min}}$ and $q_{\text{max}}$ are two constants.

The scatterer is illuminated by a plane incident field
\begin{align}
u^{\text{inc}}(\mathbf{r}) = e^{i\kappa\mathbf{r}\cdot\mathbf{d}},
\end{align}
where $\mathbf{d} = (\cos\theta, \sin\theta) \in \{ \mathbf{r}\in\mathbb{R}^{2} \, : \, |\mathbf{r}| = 1 \}$ is the incident direction
and $\theta \in (0,2\pi)$ is the incident angle. Obviously, the incident field satisfies
\begin{align}\label{inEq}
\Delta u^{\text{inc}} + \kappa^{2} u^{\text{inc}} = 0 \quad \text{in }\mathbb{R}^{2}.
\end{align}

The total field $u$ consists of the incident field $u^{\text{inc}}$ and the scattered field $u^{s}$
\begin{align}\label{fenjieEq}
u = u^{\text{inc}} + u^{s}.
\end{align}
It follows from (\ref{zongEq}), (\ref{inEq}) and (\ref{fenjieEq}) that the scattered field satisfies
\begin{align}\label{scatterEq}
\Delta u^{s} + \kappa^{2}(1+q)u^{s} = -\kappa^{2}qu^{\text{inc}} \quad \text{in }\mathbb{R}^{2},
\end{align}
accompanied with the following Sommerfeld radiation condition
\begin{align}\label{radiatEq}
\lim_{|\mathbf{r}| \rightarrow \infty}r^{1/2}\big(\partial_{r}u^{s} - i\kappa u^{s}\big) = 0,
\end{align}
where $r = |\mathbf{r}|$.


\subsection{Well-posedness in the sense of Bayesian formulation}\label{WellSubsec}

In this subsection, we suppose that the scatterer $q(\cdot)$ appeared in (\ref{zongEq}) has compact support
and $\text{supp}(q) \subset \Omega \subset B_{R}$ where $\Omega$ is a square region.
For the reader's convenience, we provide an illustration of this relation in Figure \ref{illuFig}.
\begin{figure}[htbp]
  \centering
  \includegraphics[width = 0.35\textwidth]{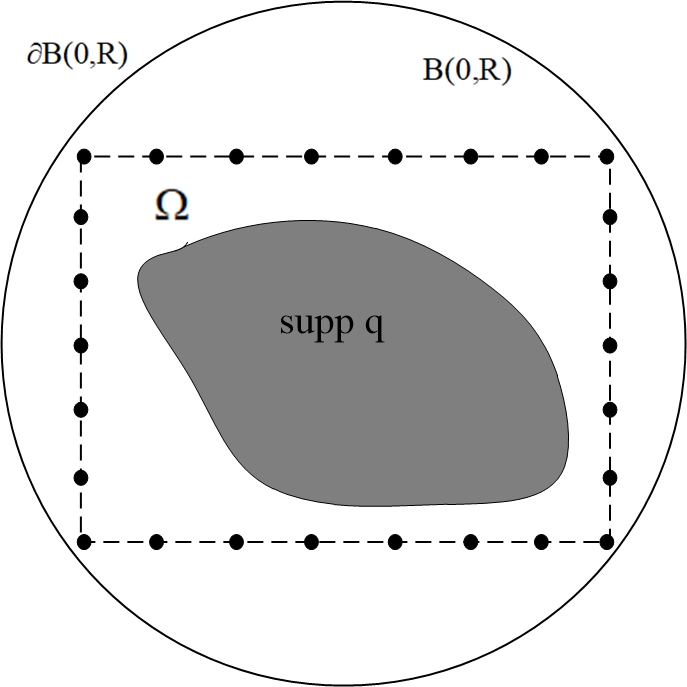}\\
  \caption{Geometry of the scattering problem.}\label{illuFig}
\end{figure}

Because the scatterer $q(\cdot)$ is assumed to have compact support, the
problem (\ref{scatterEq}) and (\ref{radiatEq}) defined on $\mathbb{R}^{2}$ can be reformulated to the following problem
defined on bounded domain \cite{Bao2015TopicReview}
\begin{align}\label{BoundedHelEq}
\left \{\begin{aligned}
& \Delta u^{s} + \kappa^{2}(1+q)u^{s} = -\kappa^{2}qu^{\text{inc}} \quad \text{in }B_{R}, \\
& \partial_{\mathbf{n}}u^{s} = \mathcal{T}u^{s} \quad \text{on }\partial B_{R},
\end{aligned}\right.
\end{align}
where $\mathcal{T}$ is the Dirichlet-to-Neumann (DtN) operator defined as follows: for any $\varphi \in H^{1/2}(\partial B_{R})$,
\begin{align}
(\mathcal{T}\varphi)(R,\theta) = \kappa\sum_{n\in\mathbb{Z}}\frac{H^{(1)'}_{n}(\kappa R)}{H^{(1)}_{n}(\kappa R)}\hat{\varphi}_{n}e^{in\theta}
\end{align}
with $H^{(1)}_{n}$ is the Hankel function of the first kind with order $n$ and
\begin{align*}
\hat{\varphi}_{n} = (2\pi)^{-1}\int_{0}^{2\pi} \varphi(R,\theta)e^{-in\theta}d\theta.
\end{align*}

For problem (\ref{BoundedHelEq}), we define the map $\mathcal{S}(q,\kappa)u^{\text{inc}}$
by $u^{s} = \mathcal{S}(q,\kappa)u^{\text{inc}}$ as in \cite{Bao2015TopicReview}.
From \cite{Bao2010StochasticSource,ColtonThirdBook}, we easily know that the following estimate holds for equations (\ref{BoundedHelEq})
\begin{align}\label{EstimateH}
\|u^{s}\|_{H^{2}(\Omega)} \leq C \|q\|_{L^{\infty}(\Omega)}\|u^{\text{inc}}\|_{L^{2}(B(0,R))}.
\end{align}
Considering Sobolev embedding theorem, we can define the following measurement operator
\begin{align}\label{MeaOp1}
\mathcal{M}(\mathcal{S}(q,\kappa)u^{\text{inc}})(x) = \big( u^{s}(x_{1}), \ldots, u^{s}(x_{N_{d}}) \big)^{T},
\end{align}
where $x_{i} \in \partial\Omega$, $i = 1,2,\ldots,N_{d}$, are the points where the wave field $u^{s}$ is measured.

In practice, we employ a uniaxial PML technique to transform the problem defined on the whole domain to
a problem defined on a bounded rectangular domain, as seen in Figure \ref{illuFig2}.
\begin{figure}[htbp]
  \centering
  \includegraphics[width = 0.35\textwidth]{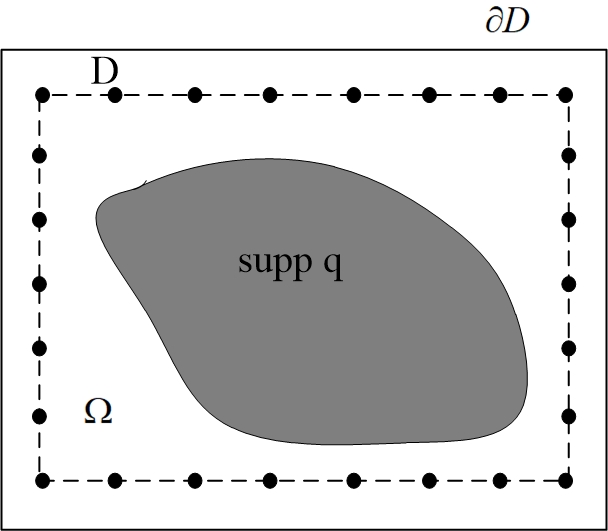}\\
  \caption{Geometry of the scattering problem with a uniaxial PML layer.}\label{illuFig2}
\end{figure}
Let $D$ be the rectangle which contain $\Omega = [x_{1},x_{2}] \times [y_{1},y_{2}]$
with $\text{supp} (q) \subset \Omega$ and let $d_{1}$ and $d_{2}$
be the thickness of the PML layers along $x$ and $y$, respectively. Let $s_{1}(x) = 1+i\sigma_{1}(x)$ and $s_{2}(y) = 1+i\sigma_{2}(y)$
be the model medium property and usually we can simply take
\begin{align*}
\sigma_{1}(x) = \left\{\begin{aligned}
& \sigma_{0}\left( \frac{x-x_{2}}{d_{1}} \right)^{p} \quad \text{for }x_{2} < x < x_{2} + d_{1} \\
& 0 \quad\quad\quad\quad\quad\quad\,\,\,\,\, \text{for }x_{1} \leq x \leq x_{2} \\
& \sigma_{0}\left( \frac{x_{1} - x}{d_{1}} \right)^{p} \quad \text{for }x_{1} - d_{1} < x < x_{1},
\end{aligned}\right.
\end{align*}
and
\begin{align*}
\sigma_{2}(y) = \left\{\begin{aligned}
& \sigma_{0}\left( \frac{y-y_{2}}{d_{2}} \right)^{p} \quad \text{for }y_{2} < y < y_{2} + d_{2} \\
& 0 \quad\quad\quad\quad\quad\quad\,\,\,\, \text{for }y_{1} \leq y \leq y_{2} \\
& \sigma_{0}\left( \frac{y_{1} - y}{d_{2}} \right)^{p} \quad \text{for }y_{1} - d_{2} < y < y_{1},
\end{aligned}\right.
\end{align*}
where the constant $\sigma_{0} > 1$ and the integer
$p \geq 2$. Denote $$s = \text{diag}(s_{2}(x)/s_{1}(x), s_{1}(x)/s_{2}(y)),$$
then the truncated PML problem can be defined as follow
\begin{align}\label{PMLBoundedHelEq}
\left \{\begin{aligned}
& \nabla\cdot(s \nabla u^{s}) + s_{1}s_{2}\kappa^{2}(1+q)u^{s} = -\kappa^{2}qu^{\text{inc}} \quad \text{in }D, \\
& u^{s} = 0 \quad \text{on }\partial D.
\end{aligned}\right.
\end{align}

Similar to the physical problem (\ref{BoundedHelEq}), we introduce the map $\mathcal{S}_{a}(q,\kappa)$
defined by $u^{s}_{a} = \mathcal{S}_{a}(q,\kappa)u^{\text{inc}}$
where $u^{s}_{a}$ stands for the solution of the truncated PML problem (\ref{PMLBoundedHelEq}).
Through similar methods for equations (\ref{BoundedHelEq}), we can prove that $u^{s}_{a}$ is a continuous function and satisfies
\begin{align}\label{EstimateHDis}
\|u^{s}_{a}\|_{L^{\infty}(D)} \leq C  \|q\|_{L^{\infty}(D)} \|u^{\text{inc}}\|_{L^{2}(D)}.
\end{align}
Now, we can define the measurement operator similar to (\ref{MeaOp1}) as follow
\begin{align}\label{MeaOp2}
\mathcal{M}(\mathcal{S}_{a}(q,\kappa)u^{\text{inc}})(x) = \big( u^{s}_{a}(x_{1}), \ldots, u^{s}_{a}(x_{N_{d}}) \big)^{T},
\end{align}
where $x_{i} \in \partial D$, $i = 1,2,\ldots,N_{d}$.

In order to introduce appropriate Gaussian probability measures, we give the following assumptions related to the covariance operator.

\textbf{Assumption 2.} Denote $A$ to be an operator, densely defined on the Hilbert space $\mathcal{H} = L^{2}(D;\mathbb{R}^{d})$,
satisfies the following properties:
\begin{enumerate}
  \item $A$ is positive-definite, self-adjoint and invertible;
  \item the eigenfunctions $\{\varphi_{j}\}_{j\in\mathbb{N}}$ of $A$, form an orthonormal basis for $\mathcal{H}$;
  \item the eigenvalues satisfy $\alpha_{j} \asymp j^{2/d}$, for all $j\in\mathbb{N}$;
  \item there is $C > 0$ such that
  \begin{align*}
  \sup_{j\in\mathbb{N}}\left( \|\varphi_{j}\|_{L^{\infty}} + \frac{1}{j^{1/d}}\text{Lip}(\varphi_{j}) \right) \leq C.
  \end{align*}
\end{enumerate}

At this moment, we can show well-posedness for inverse medium scattering problem with some Gaussian prior probability measures.
For a constant $s > 1$, we consider the prior probability measure to be a Gaussian measure $\mu_{0} := \mathcal{N}(\bar{q},A^{-s})$
where $\bar{q}$ is the mean value and the operator $A$ satisfies Assumption 2.
In addition, we take $X = C^{t}$ with $t < s - 1$. Then we know that $\mu_{0}(X) = 1$ by Example 2.19 shown in \cite{Dashti2014}.

For the scattering problem, we can take $\mathcal{F}_{a}(q) = \mathcal{M}(\mathcal{S}_{a}(q,\kappa)u^{\text{inc}})$ and let
the noise $\eta$ obeys a Gaussian mixture distribution with density function
$$\sum_{k = 1}^{K}\pi_{k}\mathcal{N}_{c}(\eta \,|\, \zeta_{k},\Sigma_{k}+\nu I).$$
Then, the measured data $d \in \mathbb{C}^{N_{d}}$ are
\begin{align}\label{DataSpecPro}
d = \mathcal{F}_{a}(q) + \eta.
\end{align}
\begin{theorem}\label{BayeTheoScatter}
For the two dimensional problem (\ref{PMLBoundedHelEq})(problem (\ref{BoundedHelEq})), if we assume space $X$, $q \sim \mu_{0}$ and
$\eta$ are specified as previous two paragraphes in this subsection. Then, the Bayesian inverse problems of recovering input
$q \in X$ of problem (\ref{PMLBoundedHelEq})(problem (\ref{BoundedHelEq})) from data $d$ given as in (\ref{DataSpecPro})
is well formulated: the posterior $\mu^{d}$ is well defined in $X$ and it is absolutely continuous with respect to $\mu_{0}$,
the Radon-Nikodym derivative is given by (\ref{DefineMuY}) and (\ref{DefineOfZd}). Moreover, there is $C = C(r)$ such that,
for all $d_{1},d_{2} \in \mathbb{C}^{N_{d}}$ with $|d_{1}|,|d_{2}| \leq r$,
\begin{align}
d_{Hell}(\mu^{d_{1}},\mu^{d_{2}}) \leq C \|d_{1} - d_{2}\|_{2}.
\end{align}
\end{theorem}
\begin{proof}
From Section \ref{BayeTheoSection}, we easily know that Theorem \ref{BayeTheoScatter} holds when Assumption 1 is satisfied.
According to the estimates (\ref{EstimateHDis}) and (\ref{MeaOp2}), we find that
\begin{align}
\|\mathcal{F}_{a}(q)\|_{2} \leq C \|q\|_{L^{\infty}(D)},
\end{align}
which indicates that statement (1) of Assumption 1 holds.
In order to verify statement (2) of Assumption 1, we denote $u_{a}^{s} + \delta u = \mathcal{F}_{a}(q+\delta q)$.
By simple calculations, we deduce that $\delta u$ satisfies
\begin{align}\label{deltaUEq}
\left \{\begin{aligned}
& \nabla\cdot(s \nabla \delta u) + s_{1}s_{2}\kappa^{2}(1+q)\delta u = -\kappa^{2}\delta q(u^{\text{inc}} + s_{1}s_{2} u_{a}^{s}) \quad \text{in }D, \\
& \delta u = 0 \quad \text{on }\partial D.
\end{aligned}\right.
\end{align}
Now, denote $\mathcal{F}_{a}'(q)$ to be the Fr\'{e}chet derivative of $\mathcal{F}_{a}(q)$, we find that
\begin{align}\label{FDerQ}
\mathcal{F}_{a}'(q)\delta q = \mathcal{M}(\delta u),
\end{align}
where $\delta u$ is the solution of equations (\ref{deltaUEq}).
By using some basic estimates for equations (\ref{PMLBoundedHelEq}), we obtain
\begin{align}\label{FDerEst}
\|\mathcal{F}_{a}'(q)\delta q\|_{2} \leq \|\delta u\|_{L^{\infty(D)}} \leq C(1+\|q\|_{L^{\infty}(\Omega)})\|\delta q\|_{L^{\infty}(D)},
\end{align}
where $C$ depends on $\kappa$, $D$, $s_{1}$ and $s_{2}$.
Estimate (\ref{FDerEst}) ensures that statement (2) of Assumption 1 holds, and the proof is completed by employing Theorem \ref{wellPosedBaye}.
\end{proof}

\begin{remark}
From the proof of Theorem \ref{BayeTheoScatter}, we can see that Theorem \ref{SmallBall} holds true for
inverse medium scattering problem considered in this subsection. Hence, we can compute the MAP estimate by minimizing
functional defined in (\ref{MiniProForm}) with the forward operator defined in (\ref{DataSpecPro}).
\end{remark}

\begin{remark}
If we assume $\mu_{0}$ is a TV-Gaussian probability measure, similar results can be obtained.
The posterior probability measure is well-defined and the MAP estimate can be obtained by solving $\min_{q\in X}J(q)$
with $J$ defined in (\ref{MiniProFormTG}). Since there are no new ingredients, we omit the details.
\end{remark}


\subsection{Learn parameters of complex Gaussian mixture distribution}\label{LearnSection}

How to estimate the parameters is one of the key steps for modeling noises by some complex Gaussian mixture distributions.
This key step consists two fundamental elements: learning examples and learning algorithms.

For the learning examples, they are the approximate errors $e := \mathcal{F}(q) - \mathcal{F}_{a}(q)$ that is
the difference of measured values for slow explicit forward solver and fast approximate forward solver.
In order to obtain this error, we need to know the unknown function $q$ which is impossible.
However, in practical problems, we usually know some prior knowledge of the unknown function $q$.
Relying on the prior knowledge, we can construct some probability measures to generate functions which we
believe to maintain similar statistical properties as the real unknown function $q$.
For this, we refer to a recent paper \cite{Iglesias2014IP}.
Since this procedure depends on specific application fields, we only provide details in Section \ref{SecNumer} for concrete numerical examples.

For the learning algorithms, expectation-maximization (EM) algorithm
is often employed in the machine learning community \cite{PR2006Book}.
Here, we need to notice that the variables are complex valued and the complex Gaussian distribution are used in our case.
This leads some differences to the classical real variable situation.

In order to provide a clear explanation, let us recall some basic relationships between complex Gaussian distributions and
real Gaussian distributions which are proved in \cite{Goodman1963Annals}.
Denote $e = (e_{1}, \ldots, e_{N_{d}})^{T}$ is a $N_{d}$-tuple of complex
Gaussian random variables. Let $\tau_{k} := \text{Re}(e_{k})$ and $\varsigma_{k} := \text{Imag}(e_{k})$
as the real and imaginary parts of $e_{k}$ with $k = 1,\ldots,N_{d}$, then define
\begin{align}\label{defTau}
\tau = (\tau_{1},\varsigma_{1},\ldots,\tau_{N_{d}},\varsigma_{N_{d}})
\end{align}
is $2N_{d}$-tuple of random variables.
From the basic theories of complex Gaussian distributions, we know that $\tau$ is $2N_{d}$-variate Gaussian distributed.
Denote the covariance matrix of $e$ by $\Sigma$ and the covariance matrix of $\tau$ by $\tilde{\Sigma}$.
As usual, we assume $\Sigma$ is a positive definite Hermitian matrix, then $\tilde{\Sigma}$ is a positive definite symmetric matrix
by Theorem 2.2 and Theorem 2.3 in \cite{Goodman1963Annals}.
In addition, we have the following lemma which is proved in \cite{Goodman1963Annals}.
\begin{lemma}\label{complexGauPro}
For complex Gaussian distributions, we have that the matrix $\Sigma$ is isomorphic to the matrix $2\tilde{\Sigma}$,
$e^{H}\Sigma e = \tau^{T}\tilde{\Sigma}\tau$ and $\text{det}(\Sigma)^{2} = \text{det}(\tilde{\Sigma})$.
\end{lemma}

Let $N_{s} \in \mathbb{N}^{+}$ stands for the number of learning examples.
Let $e_{n} = (e_{1}^{n}, \ldots, e_{N_{d}}^{n})^{T}$ with $n = 1,\ldots,N_{s}$ represent
$N_{s}$ learning examples. Then, for some fixed $K \in \mathbb{N}^{+}$,
we need to solve the following optimization problem to obtain estimations of parameters
\begin{align}
\min_{\{\pi_{k}, \zeta_{k}, \Sigma_{k}\}_{k=1}^{K}} J_{G}(\{\pi_{k}, \zeta_{k}, \Sigma_{k}\}_{k =1}^{K}),
\end{align}
where
\begin{align}
J_{G}(\{\pi_{k}, \zeta_{k}, \Sigma_{k}\}_{k =1}^{K}) := \sum_{n = 1}^{N_{s}} \ln \Bigg\{ \sum_{k = 1}^{K} \pi_{k}
\mathcal{N}_{c}(e_{n} \, | \, \zeta_{k},\Sigma_{k}) \Bigg\}.
\end{align}

In the following, we only show two different parts compared with the real variable Gaussian case.

\textbf{Estimation of means}: Setting the derivatives of $J_{G}(\{\pi_{k}, \zeta_{k}, \Sigma_{k}\}_{k =1}^{K})$ with respect to $\zeta_{k}$
of the complex Gaussian components to zero and using Lemma \ref{complexGauPro}, we obtain
\begin{align}
0 = -\sum_{n=1}^{N_{s}}\frac{\pi_{k}\mathcal{N}_{c}(e_{n}\,|\,\zeta_{k},\Sigma_{k})}
{\sum_{j = 1}^{K}\pi_{j}\mathcal{N}_{c}(e_{j}\,|\,\zeta_{j},\Sigma_{j})}\tilde{\Sigma}_{k}^{-1}(\tau_{n} - \tilde{\zeta}_{k}),
\end{align}
where $\tau_{n}$ defined as in (\ref{defTau}) with $e$ replaced by $e_{n}$,
$\tilde{\zeta}_{k}$ also defined as in (\ref{defTau}) with $e$ replaced by $\zeta_{k}$ and $\tilde{\Sigma}_{k}$
is the covariance matrix corresponding to $\Sigma_{k}$.
Hence, by some simple simplification, we find that
\begin{align}
\zeta_{k} = \frac{1}{\tilde{N}_{k}}\sum_{n=1}^{N_{s}}\gamma_{nk}e_{n},
\end{align}
where
\begin{align}\label{defineTilN}
\tilde{N}_{k} = \sum_{n=1}^{N_{s}}\gamma_{nk},  \quad
\gamma_{nk} = \frac{\pi_{k}\mathcal{N}_{c}(e_{n}\,|\,\zeta_{k},\Sigma_{k})}
{\sum_{j = 1}^{K}\pi_{j}\mathcal{N}_{c}(e_{j}\,|\,\zeta_{j},\Sigma_{j})}.
\end{align}
In the above formula, $\tilde{N}_{k}$ usually interpret as the effective number of points assigned to cluster $k$
and $\gamma_{nk}$ usually is a variable depend on latent variables \cite{PR2006Book}.

\textbf{Estimation of covariances}: For the covariances, we need to use latent variables to provide the
following complete-data log likelihood function as formula (9.40) shown in \cite{PR2006Book}
\begin{align}
\sum_{n = 1}^{N_{s}}\sum_{k = 1}^{K}\gamma_{nk}\Big\{ \ln\pi_{k} + \ln\mathcal{N}_{c}(e_{n} \, | \, \zeta_{k}, \Sigma_{k}) \Big\}.
\end{align}
Now, for $k = 1,\dots,K$, we prove that
\begin{align}
\Sigma_{k} := \frac{1}{\tilde{N}_{k}}\sum_{n = 1}^{N_{s}}\gamma_{nk}(e_{n} - \zeta_{k})(e_{n} - \zeta_{k})^{H}
\end{align}
solves the following maximization problem
\begin{align}\label{ZuidaWen}
\max_{\{\Sigma_{k}\}_{k=1}^{K}}\Bigg\{\sum_{n = 1}^{N_{s}}\sum_{k = 1}^{K}\gamma_{nk}
\Big( \ln\pi_{k} + \ln\mathcal{N}_{c}(e_{n} \, | \, \zeta_{k}, \Sigma_{k}) \Big)\Bigg\}.
\end{align}
\begin{proof}
Denote
\begin{align}
L = \sum_{n = 1}^{N_{s}}\sum_{k = 1}^{K}\gamma_{nk}
\Big( \ln\pi_{k} + \ln\mathcal{N}_{c}(e_{n} \, | \, \zeta_{k}, \Sigma_{k}) \Big).
\end{align}
Let
\begin{align}
B_{k} := \frac{1}{\tilde{N}_{k}}\sum_{n = 1}^{N_{s}}\gamma_{nk}(e_{n} - \zeta_{k})(e_{n} - \zeta_{k})^{H}
\end{align}
and notice that
\begin{align*}
\sum_{n=1}^{N_{s}}\sum_{k=1}^{K}\gamma_{nk}(e_{n}-\zeta_{k})^{H}\Sigma_{k}^{-1}(e_{n}-\zeta_{k})
& = \sum_{n=1}^{N_{s}}\sum_{k=1}^{K}\gamma_{nk}\text{tr}\Big( \Sigma_{k}^{-1}(e_{n}-\zeta_{k})(e_{n}-\zeta_{k})^{H} \Big) \\
& = \sum_{k=1}^{K}\text{tr}\Big( \Sigma_{k}^{-1}\sum_{n=1}^{N_{s}}\gamma_{nk}(e_{n}-\zeta_{k})(e_{n}-\zeta_{k})^{H} \Big) \\
& = \sum_{k=1}^{K}\tilde{N}_{k}\text{tr}\Big( \Sigma_{k}^{-1}B_{k} \Big),
\end{align*}
where $\tilde{N}_{k}$ defined as in (\ref{defineTilN}).
Then, using the explicit form of density function, we obtain
\begin{align}\label{LDEF}
L = - \sum_{k = 1}^{K}\tilde{N}_{k}\ln\text{det}(\Sigma_{k}) - \sum_{k=1}^{K}\tilde{N}_{k}\text{tr}(\Sigma_{k}^{-1}B_{k})
- N_{d}\ln\pi + \sum_{k=1}^{K}\tilde{N}_{k}\ln\pi_{k}.
\end{align}
Define $p(\xi,\Sigma):= \frac{1}{\pi^{N_{d}}\text{det}(\Sigma)}\exp\left( -\xi^{H}\Sigma^{-1}\xi \right)$, then we have
\begin{align}\label{JDEF}
\begin{split}
J & = \sum_{k=1}^{K} \tilde{N_{k}} \int_{\xi}p(\xi,\Sigma_{k}^{-1})\ln\Big( p(\xi,B_{k}^{-1})/p(\xi,\Sigma_{k}^{-1}) \Big) d\xi \\
& = \int_{\xi} \Bigg\{ \left( \ln\text{det}(B_{k}) - \xi^{H}B_{k}\xi \right)p(\xi,\Sigma_{k}^{-1}) \\
& \quad\quad\quad\quad\quad\quad\quad\quad\quad\quad
- \left( \ln\text{det}(\Sigma_{k}) - \xi^{H}\Sigma_{k}\xi \right)p(\xi,\Sigma_{k}^{-1}) \Bigg\}d\xi \\
& = \sum_{k=1}^{K}\tilde{N}_{k}\ln\text{det}(B_{k}) + \sum_{k=1}^{K}\tilde{N}_{k}\text{tr}(I) \\
& \quad\quad\quad\quad\quad\quad\quad\quad\quad
- \sum_{k=1}^{K}\tilde{N}_{k}\text{tr}(\Sigma_{k}^{-1}B_{k}) - \sum_{k=1}^{K}\tilde{N}_{k}\ln\text{det}(\Sigma_{k}),
\end{split}
\end{align}
where Corollary 4.1 in \cite{Goodman1963Annals} has been used for the last equality.
On comparing the final result of (\ref{LDEF}) with (\ref{JDEF}) one observes that any series Hermitian positive definite matrixes
$\{\Sigma_{k}\}_{k=1}^{K}$ that maximize $L$ maximize $J$ and conversely.
Now, $\ln u \leq u-1$ with equality holding if and only if $u = 1$.
Thus
\begin{align}\label{Jleq1}
\begin{split}
J & = \sum_{k=1}^{K} \tilde{N_{k}} \int_{\xi}p(\xi,\Sigma_{k}^{-1})\ln\Big( p(\xi,B_{k}^{-1})/p(\xi,\Sigma_{k}^{-1}) \Big) d\xi \\
& \leq \sum_{k=1}^{K} \tilde{N_{k}} \int_{\xi} p(\xi,\Sigma_{k}^{-1})\Big( p(\xi,B_{k}^{-1})/p(\xi,\Sigma_{k}^{-1}) - 1 \Big) d\xi = 0.
\end{split}
\end{align}
If and only if $p(\xi,\Sigma_{k}) = p(\xi,B_{k})$ with $k=1,\ldots,K$, equality in (\ref{Jleq1}) holds true.
Hence, $\Sigma_{k} = B_{k}\, (k=1,\ldots,K)$ solves problem (\ref{ZuidaWen}).
\end{proof}

With these preparations, we can easily construct EM algorithm following the line of reasoning shown in Chapter 9 of \cite{PR2006Book}.
For concisely, the details are omitted and we provide the EM algorithm in Algorithm \ref{algComplexEM}.
\begin{algorithm}
\caption{Complex EM algorithm}
\label{algComplexEM}
\begin{algorithmic}
\STATE {\textbf{Step 1}: For a series of samples $e_{n} \in \mathbb{C}^{N_{d}} \, (n = 1,\ldots,N_{s})$,
initialize the means $\zeta_{k}$, covariances $\Sigma_{k}$ and mixing coefficients $\pi_{k}$, and evaluate
the initial value of the $\ln$ likelihood.}
\STATE {\textbf{Step 2 (E step)}: Evaluate the responsibilities using the current parameter values
\begin{align*}
\gamma_{nk} = \frac{\pi_{k}\mathcal{N}_{c}(e_{n} \, | \, \zeta_{k},\Sigma_{k})
}{\sum_{j = 1}^{K} \pi_{j}\mathcal{N}_{c}(e_{n} \, | \, \zeta_{j},\Sigma_{j}).
}
\end{align*}
}
\STATE {\textbf{Step 3 (M step)}: Re-estimate the parameters using the current responsibilities
\begin{align*}
\tilde{N}_{k} = \sum_{n=1}^{N_{s}}\gamma_{nk},  \quad \pi_{k}^{\text{new}} = \frac{\tilde{N}_{k}}{N_{s}}, \quad
\zeta_{k}^{\text{new}} = \frac{1}{\tilde{N}_{k}} \sum_{n=1}^{N_{s}}\gamma_{nk}e_{n},
\end{align*}
\begin{align*}
\Sigma_{k}^{\text{new}} = \frac{1}{\tilde{N}_{k}} \sum_{n=1}^{N_{s}}\gamma_{nk}(e_{n}-\zeta_{k}^{\text{new}})
(e_{n}-\zeta_{k}^{\text{new}})^{H}
\end{align*}
}
\STATE {\textbf{Step 4}: Evaluate the $\ln$ likelihood
\begin{align*}
\sum_{n=1}^{N_{s}}\ln\Bigg\{ \sum_{k=1}^{K}\pi_{k}\mathcal{N}_{c}(e_{n} \, | \, \zeta_{k},\Sigma_{k}) \Bigg\}
\end{align*}
and check for convergence of either the parameters or the $\ln$ likelihood.
If the convergence criterion is not satisfied return to Step 2.
}
\end{algorithmic}
\end{algorithm}

\begin{remark}
In Algorithm \ref{algComplexEM}, if the parameters satisfy $N_{d} < N_{s}$, we can usually obtain nonsingular matrixes $\{\Sigma_{k}\}_{k=1}^{K}$.
However, in our case, we can not generate so many learning examples $N_{s}$ and the number of measuring points $N_{d}$ is usually very large
for real world applications.
Hence, we will meet the situation $N_{d} > N_{s}$ which makes $\{\Sigma_{k}\}_{k=1}^{K}$ to be a series of singular matrixes.
In order to solve this problem, we adopt a simple strategy that is replace the estimation of $\Sigma_{k}$ in Step 3 by
the following formula
\begin{align}\label{MStepReg}
\Sigma_{k}^{\text{new}} = \frac{1}{\tilde{N}_{k}} \sum_{n=1}^{N_{s}}\gamma_{nk}(e_{n}-\zeta_{k}^{\text{new}})
(e_{n}-\zeta_{k}^{\text{new}})^{H} + \delta I,
\end{align}
where $\delta$ is a small positive number named as the regularization parameter.
\end{remark}


\subsection{Adjoint state approach with model error compensation}

By Algorithm \ref{algComplexEM}, we obtain the estimated mixing coefficients, mean values and covariance matrixes.
From the statements shown in Section \ref{BayeTheoSection} and Subsection \ref{WellSubsec}, it is obviously that we need to solve
optimization problems as follows
\begin{align}
\min_{q\in L^{\infty}(\Omega)} \Big\{ - \Phi(q;d) + \mathcal{R}(q) \Big\},
\end{align}
where
\begin{align}
- \Phi(q;d)\! = \!  - \ln\Bigg\{ \sum_{k=1}^{K}\pi_{k}\frac{1}{\pi^{N_{d}}\det(\Sigma_{k}+\nu I)}
\exp\Big( -\frac{1}{2}\Big\| d-\mathcal{F}_{a}(q)-\zeta_{k} \Big\|_{\Sigma_{k}+\nu I}^{2} \Big) \Bigg\}, \\
\mathcal{R}(q) = \frac{1}{2}\|A^{s/2}q\|_{L^{2}(\Omega)}^{2}\, \quad \text{or} \quad
\mathcal{R}(q) = \lambda \|q\|_{\text{TV}} + \frac{1}{2}\|A^{s/2}q\|_{L^{2}(\Omega)}^{2}.  \label{DefFunR}
\end{align}
Different form of functional $\mathcal{R}$ comes from different assumptions of the prior probability measures:
Gaussian probability measure or TV-Gaussian probability measure.
For the multi-frequency approach of inverse medium scattering problem, the forward operator in each optimization problem
is related to $\kappa$. So we rewrite $\mathcal{F}_{a}(q)$ and $\Phi(q;d)$ as $\mathcal{F}_{a}(q,\kappa)$ and $\Phi(q,\kappa;d)$,
which emphasize the dependence of $\kappa$.
We have a series of wavenumbers $0 < \kappa_{1} < \kappa_{2} < \cdots \kappa_{N_{w}} < \infty$,
and we actually need to solve a series optimization problems
\begin{align} \label{opt1}
\min_{q\in L^{\infty}(\Omega)} \Big\{ - \Phi(q,\kappa_{i};d) + \mathcal{R}(q) \Big\}
\end{align}
with $i$ from $1$ to $N_{w}$ and the solution of the previous optimization problem is the initial data for the later optimization problem.

Denote $F(q) = - \Phi(q,\kappa_{i};d)$. To minimize
the cost functional by a gradient method, it is required to compute Fr\'{e}chet derivative of functionals $F$ and $\mathcal{R}$.
For functional $\mathcal{R}$ with form shown in (\ref{DefFunR}), we can obtain the Fr\'{e}chet derivatives as follows
\begin{align}\label{DerR1}
\mathcal{R}'(q) = A^{s}q, \quad \text{or} \quad \mathcal{R}'(q) = A^{s}q + 2\lambda \nabla\cdot\left( \frac{\nabla q}{\sqrt{|\nabla q|^{2} + \delta}} \right),
\end{align}
where we used the following modified version of $\mathcal{R}$
\begin{align}
\mathcal{R}(q) = \lambda\int_{\Omega}\sqrt{|\nabla q|^{2}+\delta} + \frac{1}{2}\|A^{s/2}q\|_{L^{2}(\Omega)}^{2}
\end{align}
for the TV-Gaussian prior case and $\delta$ is a small smoothing parameter avoiding zero denominator in (\ref{DerR1}).

Next, we consider the functional $F$ with $\mathcal{F}_{a}$ is the forward operator related to problem (\ref{PMLBoundedHelEq}).
A simple calculation yields the derivative of $F$ at $q$;
\begin{align}\label{zuihou0}
F'(q)\delta q = \text{Re} \Big( \mathcal{M}(\delta u), \sum_{k = 1}^{K} \gamma_{k} (\Sigma_{k}+\nu I)^{-1}(d-\mathcal{F}_{a}(q,\kappa_{i})-\zeta_{k}) \Big),
\end{align}
where $\delta u$ satisfy
\begin{align}\label{deltaUEq2}
\left \{\begin{aligned}
& \nabla\cdot(s \nabla \delta u) + s_{1}s_{2}\kappa_{i}^{2}(1+q)\delta u = -\kappa^{2}\delta q(u^{\text{inc}} + s_{1}s_{2}u_{a}^{s}) \quad \text{in }D, \\
& \delta u = 0 \quad \text{on }\partial D,
\end{aligned}\right.
\end{align}
and
\begin{align*}
\gamma_{k} = \frac{\pi_{k}\mathcal{N}_{c}(d-\mathcal{F}_{a}(q) \, | \, \zeta_{k},\Sigma_{k}+\nu I)
}{\sum_{j = 1}^{K} \pi_{j}\mathcal{N}_{c}(d-\mathcal{F}_{a}(q) \, | \, \zeta_{j},\Sigma_{j}+\nu I).
}
\end{align*}
To compute the Fr\'{e}chet derivative, we introduce the adjoint system:
\begin{align}\label{AdjSystem}
\left \{\begin{aligned}
& \nabla\cdot(\bar{s}\nabla v) + \bar{s}_{1}\bar{s}_{2}\kappa_{i}^{2}(1+q) v =
- \kappa_{i}^{2} \sum_{j = 1}^{N_{d}}\delta(x-x_{j})\rho_{j} \quad \text{in }D, \\
& v = 0 \quad \text{on }D,
\end{aligned}\right.
\end{align}
where $\rho_{j} \, (j=1,\ldots,N_{d})$ denote the $j$th component of
$\sum_{k = 1}^{K} \gamma_{k} (\Sigma_{k}+\nu I)^{-1}(d-\mathcal{F}_{a}(q,\kappa_{i})-\zeta_{k}) \in \mathbb{C}^{N_{d}}$.
Multiplying equation (\ref{deltaUEq2}) with the complex conjugate of $v$ on both sides and integrating over $D$ yields
\begin{align*}
\int_{D}\nabla\cdot(s \nabla \delta u)\bar{v} + s_{1}s_{2}\kappa_{i}^{2}(1+q)\delta u \bar{v}
= - \int_{D}\kappa^{2}\delta q(u^{\text{inc}} + s_{1}s_{2}u_{a}^{s})\bar{v}.
\end{align*}
By integration by parts formula, we obtain
\begin{align*}
\int_{D}\delta u \Big( \nabla\cdot(s \nabla \bar{v}) + s_{1}s_{2}\kappa_{i}^{2}(1+q)\bar{v} \Big)
= - \kappa_{i}^{2}\int_{D}\delta q (u^{\text{inc}} + s_{1}s_{2}u_{a}^{s})\bar{v}.
\end{align*}
Taking complex conjugate of equation (\ref{AdjSystem}) and plugging into the above equation yields
\begin{align*}
-\kappa_{i}^{2}\int_{D}\delta u \sum_{j = 1}^{N_{d}}\delta(x-x_{j})\bar{\rho}_{j}
= - \kappa_{i}^{2}\int_{D}\delta q (u^{\text{inc}} + s_{1}s_{2}u_{a}^{s})\bar{v},
\end{align*}
which implies
\begin{align}\label{zuihou1}
\Big( \mathcal{M}(\delta u),  \sum_{k = 1}^{K} \gamma_{k} \Sigma_{k}^{-1}(d-\mathcal{F}_{a}(q,\kappa_{i})-\zeta_{k}) \Big)
= \int_{D}\delta q (u^{\text{inc}} + s_{1}s_{2}u_{a}^{s})\bar{v}.
\end{align}
Considering both (\ref{zuihou0}) and (\ref{zuihou1}), we find that
\begin{align*}
F'(q)\delta q = \text{Re}\int_{D}\delta q (u^{\text{inc}} + s_{1}s_{2}u_{a}^{s})\bar{v},
\end{align*}
which gives the Fr\'{e}chet derivative as follow
\begin{align}\label{Fdd1}
F'(q) = \text{Re}\big( (\bar{u}^{\text{inc}} + \bar{s}_{1}\bar{s}_{2}\bar{u}_{a}^{s})v \big).
\end{align}

With these preparations, it is enough to construct Gaussian mixture recursive linearization method (GMRLM) which is shown in Algorithm \ref{alg23}.
Notice that for the recursive linearization method (RLM) shown in \cite{Bao2015TopicReview},
only one iteration of the gradient descent method for each fixed wavenumber can provide an acceptable recovery function.
So we only iterative once for each fixed wavenumber.
\begin{algorithm}
\caption{Gaussian mixture recursive linearization method (GMRLM)}
\label{alg23}
\begin{algorithmic}
\REQUIRE {Initialize parameters: $\sigma_{0}$, $d_{1}$, $d_{2}$, $p$, $\{\zeta_{k}\}_{k=1}^{K}$, $\{\Sigma_{k}\}_{k=1}^{K}$, $\{\pi_{k}\}_{k=1}^{K}$, $q$,
wavenumbers $(\kappa_{1},\ldots, \kappa_{N_{w}})$ and
incident angles $(\textbf{d}_{1},\ldots,\textbf{d}_{N_{m}})$.}
\STATE {\textbf{Iteration}:
for $i = 1,2,\ldots, N_{w}$ \\
\hspace*{2.2 cm} for $j = 1,2,\ldots, N_{m}$ \\
\hspace*{2.7 cm} solve one forward problem (\ref{PMLBoundedHelEq}) with $\kappa = \kappa_{i}$ and $\textbf{d} = \textbf{d}_{j}$; \\
\hspace*{2.7 cm} solve one adjoint problem (\ref{AdjSystem}) with $\kappa = \kappa_{i}$ and $\textbf{d} = \textbf{d}_{j}$; \\
\hspace*{2.7 cm} compute the Fr\'{e}chet derivative by formulas (\ref{DerR1}) and (\ref{Fdd1}); \\
\hspace*{2.7 cm} update the scatterer function; \\
\hspace*{2.2 cm} end for  \\
\hspace*{1.7 cm} end for \\
}
\ENSURE {Final estimation of $q$.}
\end{algorithmic}
\end{algorithm}


\section{Numerical examples}\label{SecNumer}

In this section, we provide two numerical examples in two dimensions to illustrate the effectiveness of the proposed method.
In the following, we assume that $\Omega = \{x\in\mathbb{R}^{2} \, : \, \|x\|_{2} \leq 1\}$ with $\Omega \subset D$
where $D$ is the PML domain with $d_{1} = d_{2} = 0.15$, $p = 2.5$ and $\sigma_{0} = 1.5$.
For the forward solver, finite element method (FEM) has been employed and
the scattering data are obtained by numerical solution of the forward scattering problem with adaptive mesh technique.
For the following two examples, we choose $N_{w} = 20$ and $\textbf{d}_{j} \, (j = 1,\ldots,N_{w})$
are equally distributed around $\partial D$. Equally spaced wavenumbers are used, starting from the lowest wavenumber $\kappa_{\text{min}} = \pi$
and ending at the highest wavenumber $\kappa_{\text{max}} = 10\pi$.
Denote by $\Delta\kappa = (\kappa_{\text{max}} - \kappa_{\text{min}})/9 = \pi$ the step size of the wavenumber;
then the ten equally spaced wavenumbers are $\kappa_{j} = j\Delta\kappa$, $j = 1,\ldots,10$.
We set $400$ receivers that equally spaced along the boundary of $\Omega$ as shown in Figure \ref{illuFig2}.
For the initial guess of the unknown function $q$, there are numerous strategies, i.e.,
methods based on Born approximation \cite{Bao2015TopicReview, Bleistein2001Book}.
Since the main point here is not on the initial gauss, we just set the initial $q$ to be a function always equal to zero for simplicity.

In order to show the stability of the proposed method, some relative random noise is added to the data, i.e.,
\begin{align}
u^{s}|_{\partial \Omega} := (1+\sigma \text{rand})u^{s}|_{\partial_{D}}.
\end{align}
Here, rand gives uniformly distributed random numbers in $[-1,1]$ and $\sigma$ is a noise level parameter taken to be $0.02$
in our numerical experiments. Define the relative error by
\begin{align}
\text{Relative Error} = \frac{\|q-\tilde{q}\|_{L^{2}(\Omega)}}{\|q\|_{L^{2}(\Omega)}},
\end{align}
where $\tilde{q}$ is the reconstructed scatterer and $q$ is the true scatterer.

\textbf{Example 1}: For the first example, let
\begin{align*}
\tilde{q}(x,y) = 0.3 (1-x)^{2}e^{-x^{2}-(y+1)^{2}} - (0.2x - x^{3} - y^{5})e^{-x^{2}-y^{2}} - 0.03 e^{-(x+1)^{2} - y^{2}}
\end{align*}
and reconstruct a scatterer defined by
\begin{align*}
q(x,y) = \tilde{q}(3x,3y)
\end{align*}
inside the unit square $\{x\in\mathbb{R}^{2} \, : \, \|x\|_{2} < 1\}$.

\begin{figure}[htbp]
  \centering
  \includegraphics[width=1\textwidth]{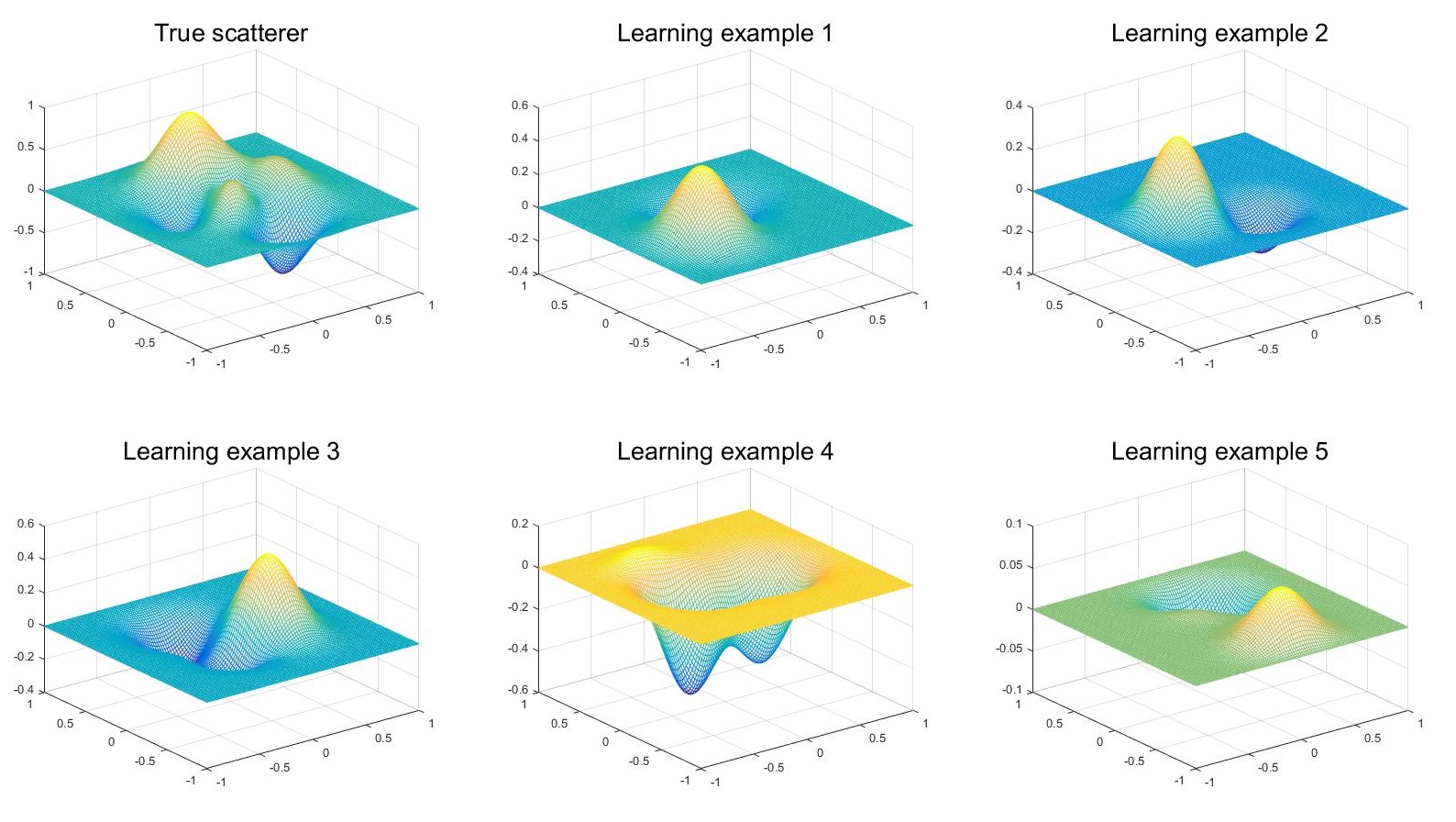}\\
  \caption{True scatterer and five typical learning examples}\label{TrueLearnFig1}
\end{figure}

Denote $U[b_{1},b_{2}]$ to be a uniform distribution with minimum value $b_{1}$ and maximum value $b_{2}$.
Now, we assume that some prior knowledge of this function $q$ have been known.
According to the prior knowledge, we generate $200$ learning examples according to the following function
\begin{align}
q_{e}(x,y) := \sum_{k = 1}^{3}(1-x^{2})^{a_{k}^{1}}(1-y^{2})^{a_{k}^{2}} a_{k}^{3} \exp\bigg(-a_{k}^{4}(x - a_{k}^{5})^{2}
- a_{k}^{6}(y - a_{k}^{7})^{2}\bigg),
\end{align}
where
\begin{align*}
& a_{k}^{1}, a_{k}^{2} \sim U[1,3], \quad a_{k}^{3} \sim U[-1,1], \\
& a_{k}^{4}, a_{k}^{6} \sim U[8,10], \quad a_{k}^{5}, a_{k}^{7} \sim U[-0.8,0.8].
\end{align*}

In order to provide an intuitional sense, we show the true scatterer and several learning examples in Figure \ref{TrueLearnFig1}.
We use $409780$ elements to obtain accurate solutions which we recognized as $\mathcal{S}(q)u^{\text{inc}}$.
To test our approach, $16204$ elements will be used to obtain $\mathcal{S}_{a}(q)u^{\text{inc}}$.

Learning algorithm with $K = 4$ proposed in Subsection \ref{LearnSection} has been used to learn the statistical properties of differences
$e_{n}^{i} := \mathcal{F}(q_{n},\kappa_{i}) - \mathcal{F}_{a}(q_{n},\kappa_{i})$ with $\kappa_{i} = i\cdot \pi \, (i = 1,\ldots 10)$
and $q_{n} \, (n = 1,\ldots 200)$ stands for the learning examples.
Concerning the regularizing term, we take $A = 0.01 \Delta$, $s = 1.5$ and $\lambda = 0$, which can be computed by Fourier transform.
Since regularization is not the main point of our paper, we will not discuss the strategies of choosing $A$ in details.

\begin{figure}[htbp]
  \centering
  \includegraphics[width=0.5\textwidth]{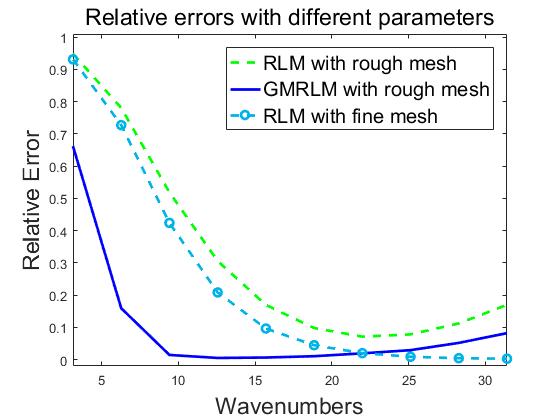}\\
  \caption{Relative errors with different parameters:
  green dotted line are relative errors obtained by using the RLM with 16204 elements;
  cyan dotted line with circles are relative errors obtained by using the RLM with 183198 elements;
  blue solid line are relative errors obtained by using the GMRLM with 16204 elements.}\label{RelaErrEx1}
\end{figure}
\begin{table}[htbp]\label{table1}
\begin{center}
\caption{Comparison of the RLM and the GMRLM with different parameters}
\begin{tabular}{c|c|c|c}
  \Xhline{1pt}
  Algorithm  & Element Number & Wavenumber & Relative Error  \\
  \hline
  RLM        & 16204          & $7\pi$     &  7.10\%    \\
  \hline
  RLM        & 183198         & $10\pi$    &  0.26\%  \\
  \hline
  GMRLM      & 16204          & $4\pi$     &  0.49\%   \\
  \hline
  RLM        & 16204          & $4\pi$     &  30.58\%   \\
  \hline
  RLM        & 183198         & $4\pi$     &  20.82\%    \\
  \hline
  RLM        & 183198         & $9\pi$     &  0.42\%    \\
  \Xhline{1pt}
\end{tabular}
\end{center}
\end{table}

Relative errors of RLM with small element number, RLM with large element number and GMRLM
with small element number have been shown in Figure \ref{RelaErrEx1}, which illustrate the effectiveness of the proposed method.
For the case of small element number, the RLM diverges when $\kappa\approx 7\pi$.
The reason is that large number of elements are needed to ensure the convergence of finite element methods for
Helmholtz equations with high wavenumber. Our error compensation method can not eliminate such errors, so
it is also diverges when $\kappa\approx 7\pi$. However, when $\kappa \approx 4\pi$, our method provides a recovered function with relative error
comparable to the result obtained by the RLM with more than eleven times of elements and $\kappa \approx 9\pi$.
Hence, by learning process, the GMRLM can give an acceptable recovered function much faster than the traditional RLM.

\begin{figure}[htbp]
  \centering
  \includegraphics[width=1\textwidth]{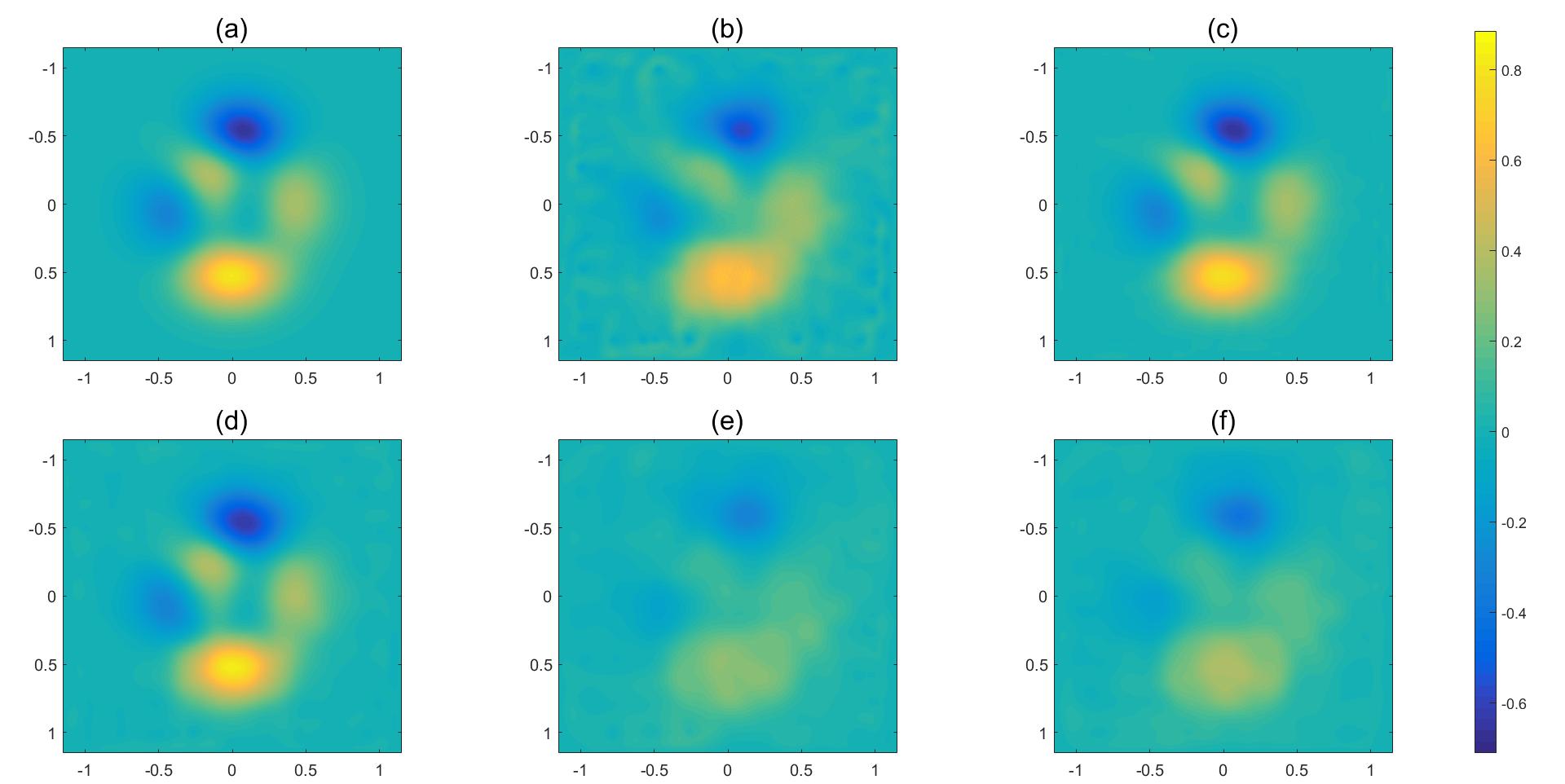}\\
  \caption{Recovered functions with different parameters. (a): true function;
  (b): minimum relative error estimate for the RLM with 16204 elements and the wavenumber computed to $7\pi$;
  (c): minimum relative error estimate for the RLM with 183198 elements and the wavenumber computed to $10\pi$;
  (d): minimum relative error estimate for the GMRLM with 16204 elements and the wavenumber computed to $4\pi$;
  (e): recovered function for the RLM with 16204 elements and the wavenumber computed to $4\pi$;
  (f): recovered function for the RLM with 183198 elements and the wavenumber computed to $4\pi$.}\label{PcolorComEx1}
\end{figure}

In addition, we show the accurate values of relative errors and element numbers in Table \ref{table1}.
In Figure \ref{PcolorComEx1}, the true scatterer function has been shown on the top left and five results obtained
by RLM and GMRLM with different parameters have been given.
From these, we can visually see the effectiveness of the proposed method.

\textbf{Example 2}:
For the second example, let
\begin{align}
q(x,y) := \left \{\begin{aligned}
& 0.7 \qquad \text{for } -0.3 \leq x \leq 0.3 \text{ and }-0.3 \leq y \leq 0.3 \\
& -0.1 \quad\! \text{for } -0.1 < x < 0.1 \text{ and }-0.1 < y < 0.1 \\
& 0 \qquad\,\,\,\,\, \text{other areas in square } -1\leq x\leq 1 \text{ and } -1\leq y\leq 1.
\end{aligned}\right.
\end{align}
As in Example 1, we need to develop some learning examples. Here, we assume that there is a square in $[-1,1]^{2}$,
but we did not know the position, size and height of the square.
We assume that the position, size and height are all uniform random variables with height between $[-1,1]$ and
the square supported in $[-1,1]^{2}$. As in Example 1, we generate 200 learning examples.
To give the reader an intuitive idea, we show the true scatterer and five typical learning examples in Figure \ref{TrueLearnFig2}.

\begin{figure}[htbp]
  \centering
  \includegraphics[width=1\textwidth]{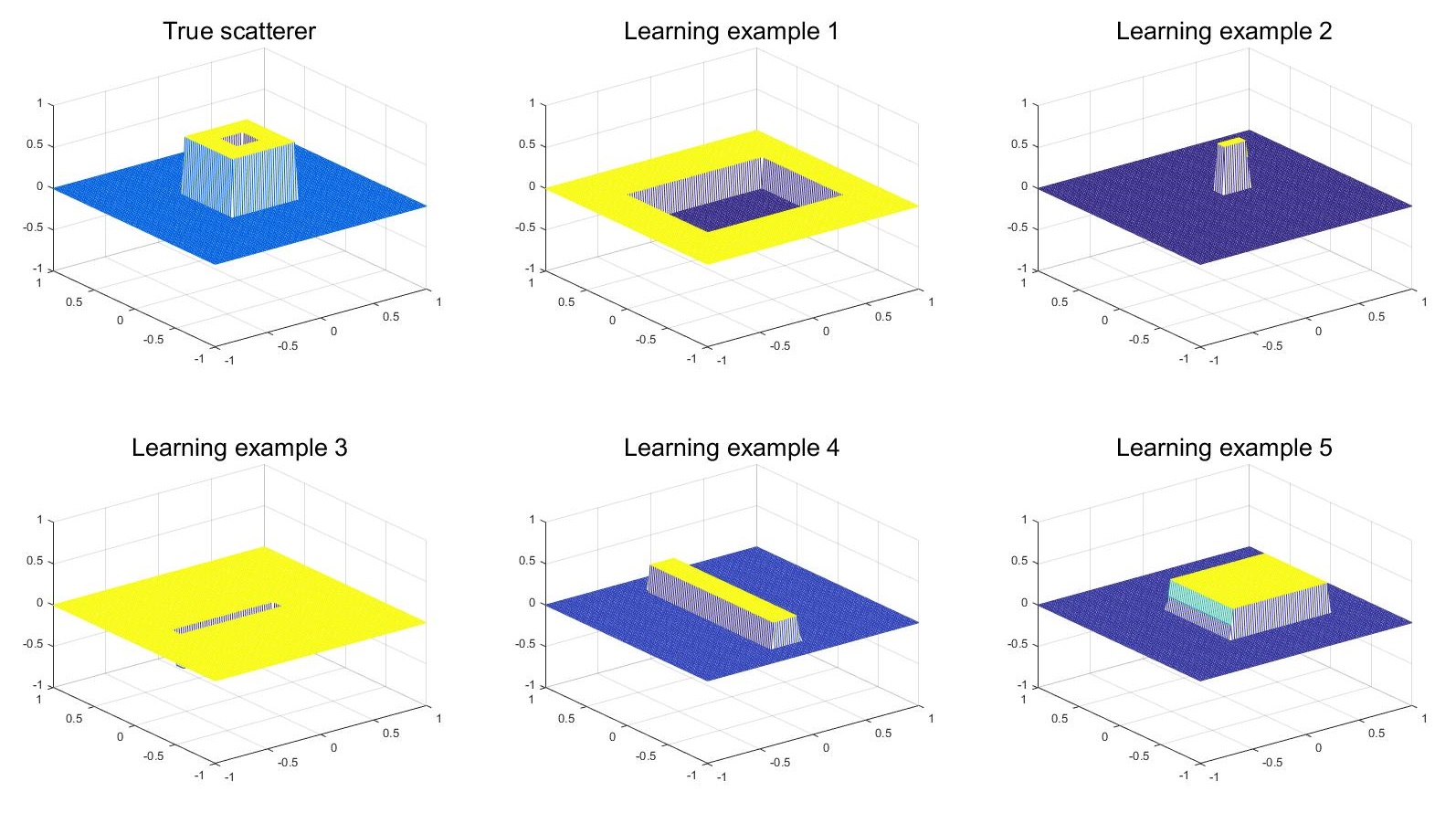}\\
  \caption{True scatterer and five typical learning examples}\label{TrueLearnFig2}
\end{figure}

\begin{figure}[htbp]
  \centering
  \includegraphics[width=0.5\textwidth]{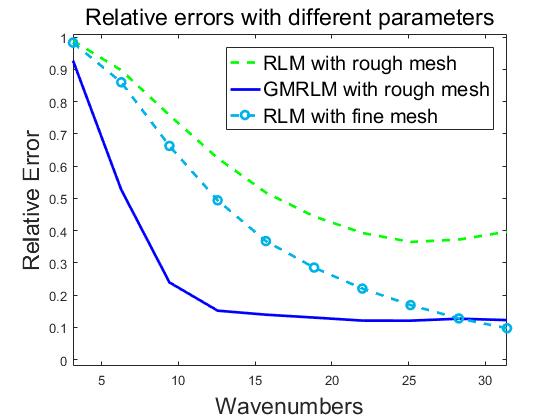}\\
  \caption{Relative errors with different parameters:
  green dotted line are relative errors obtained by using the RLM with 16204 elements;
  cyan dotted line with circles are relative errors obtained by using the RLM with 183198 elements;
  blue solid line are relative errors obtained by using the GMRLM with 16204 elements.}\label{RelaErrEx2}
\end{figure}

For this discontinuous scatterer, we take same values of parameters as in Example 1.
Beyond our expectation, the proposed algorithm obviously converges even faster than the RLM with more than eleven times of elements,
which is shown in Figure \ref{RelaErrEx2}.
By our understanding, the reason for such fast convergence is that the means and covariances learned by complex EM algorithm
not only compensate numerical errors but also encode some prior information of the true scatterer by learning examples.
Until the wavenumber is $9\pi \approx 28.26$, the RLM with 183198 elements provide a recovered function which has similar relative error
as the recovered function obtained by the GMRLM.
The RLM with only 16204 elements diverges as in Example 1 when wavenumber is too large, and the proposed algorithm
still can not compensate the loss of physics as shown in Figure \ref{RelaErrEx2}.
For accurate value of relative errors and elements, we show them in Table \ref{table2}.

\begin{table}[htbp]\label{table2}
\begin{center}
\caption{Comparison of the RLM and the GMRLM with different parameters}
\begin{tabular}{c|c|c|c}
  \Xhline{1pt}
  Algorithm  & Element Number & Wavenumber & Relative Error  \\
  \hline
  RLM        & 16204          & $8\pi$     &  36.49\%    \\
  \hline
  RLM        & 183198         & $10\pi$    &  9.71\%  \\
  \hline
  GMRLM      & 16204          & $5\pi$     &  13.92\%    \\
  \hline
  GMRLM      & 16204          & $7\pi$     &  12.09\%   \\
  \hline
  RLM        & 16204          & $7\pi$     &  39.28\%   \\
  \hline
  RLM        & 183198         & $7\pi$     &  22.03\%    \\
  \hline
  RLM        & 183198         & $9\pi$     &  12.76\%   \\
  \Xhline{1pt}
\end{tabular}
\end{center}
\end{table}

\begin{figure}[htbp]
  \centering
  \includegraphics[width=1\textwidth]{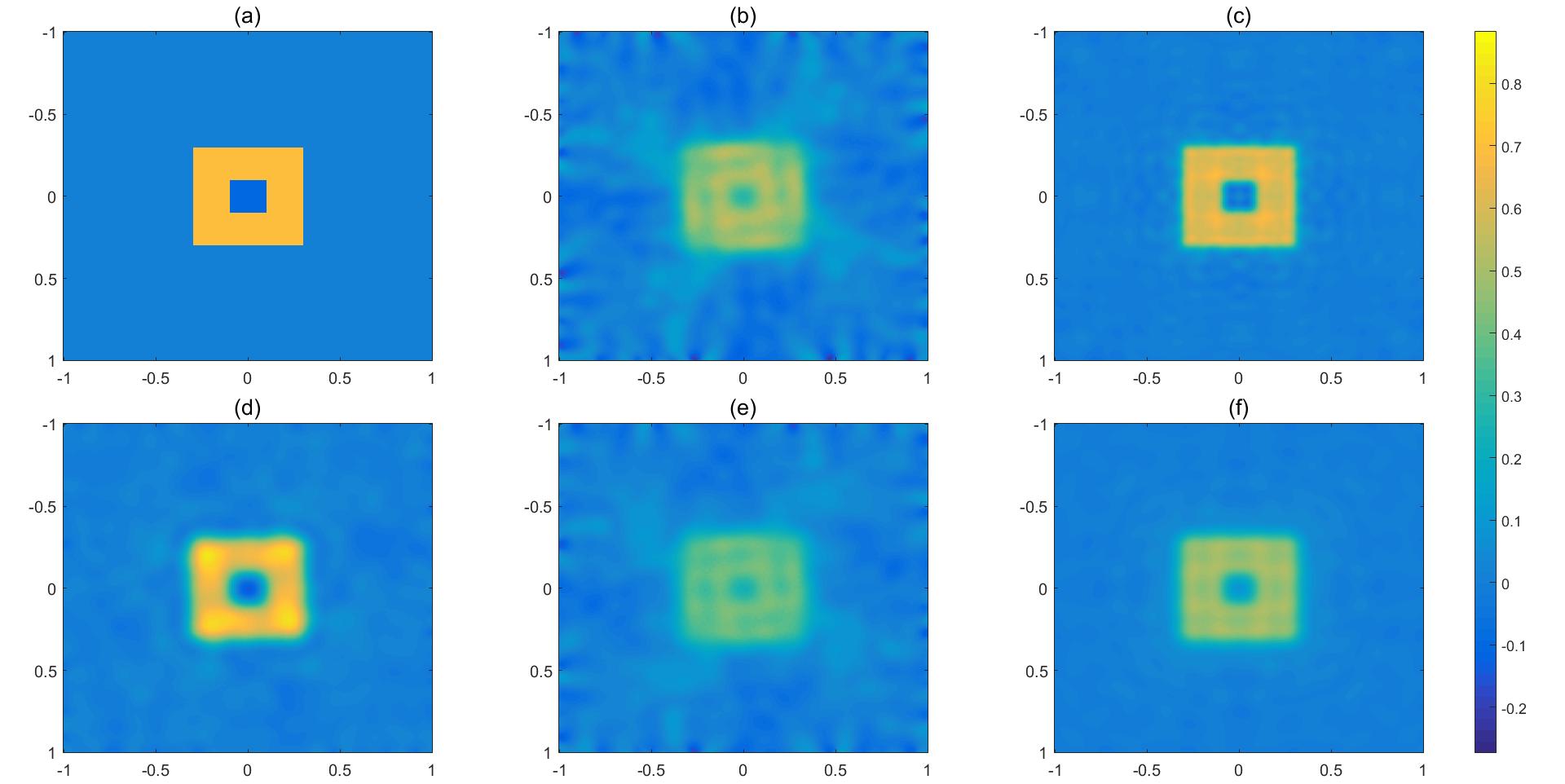}\\
  \caption{Recovered functions with different parameters. (a): true function;
  (b): minimum relative error estimate for the RLM with 16204 elements and the wavenumber computed to $8\pi$);
  (c): minimum relative error estimate for the RLM with 183198 elements and the wavenumber computed to $10\pi$;
  (d): minimum relative error estimate for the GMRLM with 16204 elements and the wavenumber computed to $8\pi$;
  (e): recovered function for the RLM with 16204 elements and the wavenumber computed to $8\pi$;
  (f): recovered function for the RLM with 183198 elements and the wavenumber computed to $8\pi$.}\label{PcolorComEx2}
\end{figure}

Finally, we provide the image of true scatterer on the top left in Figure \ref{PcolorComEx2}.
On the top middle, the best result obtained by the RLM with 16204 elements is given.
From this image, we can see that it is failed to recover the small square embedded in the large square.
The best result obtained by the RLM with 183198 elements is shown on the top right. It is much much better
than the function obtained by algorithm with 16204 elements.
At the bottom of Figure \ref{PcolorComEx2}, we show the best result obtained by the GMRLM with 16204 elements on the left
and show the results obtained by the RLM (compute to the same wavenumber as the GMRLM)
with 16204 elements and 183198 elements in the middle and on the righthand side respectively.
The recovered function by the GMRLM is not as well as the recovered function obtained by the RLM
with more than eleven times of elements and higher wavenumber.
However, beyond our expectation, it is already capture the small square embedded in
the large square, which is not incorporated in our 200 learning examples.

In summary, the proposed GMRLM converges much faster than the classical RLM and it can provide
a much better result at the same discrete level compared with the RLM.


\section{Conclusions}

In this paper, we assume the modeling errors brought by rough discretization to be Gaussian mixture random variables.
Based on this assumption, we construct the general Bayesian inverse framework and prove the relations between MAP estimates and regularization methods.
Then, the general theory has been applied to a specific inverse medium scattering problem. Well-posedness in the statistical sense
has been proved and the related optimization problem has been obtained.
In order to acquire estimates of parameters in the Gaussian mixture distribution, we generalize the EM algorithm with real variables
to the complex variables case rigorously, which incorporate the machine learning process into the classical inverse medium problem.
Finally, the adjoint problem has been deduced and the RLM has been generalized to GMRLM based on the previous illustrations.
Two numerical examples are given, which demonstrate the effectiveness of the proposed methods.

This work is just a beginning, and there are a lot of problems need to be solved.
For example, we did not give a principle of choosing parameter $K$ appeared in the Gaussian mixture distribution.
In addition, in order to learn the model errors more accurately,
we can attempt to design new algorithms to adjust the parameters in the Gaussian mixture distribution efficiently in the inverse iterative procedure.


\section*{Acknowledgments}
This work was partially supported by the NSFC under grant
Nos. 11501439, 11771347, 91730306, 41390454
and partially supported by the Major projects of the NSFC under grant Nos. 41390450 and 41390454,
and partially supported by the postdoctoral science foundation project of China under grant no. 2017T100733.

\bibliographystyle{plain}
\bibliography{references}

\end{document}